\documentclass{amsart}

\usepackage[margin=2.4cm]{geometry}
\usepackage{amssymb,amsmath,upgreek,mathrsfs,bm,cases,latexsym,hyperref,color,enumerate,graphicx,algorithm,algorithmic,placeins}
\graphicspath{{../exp/results/}{./exp/results/}}

\hypersetup{hypertexnames=false,colorlinks,linkcolor={blue},citecolor={blue},urlcolor={blue}}

\DeclareMathOperator{\proj}{proj}
\newcommand{\dotproduct}[1]{\left\langle#1\right\rangle}
\newcommand{\ud}{\mathrm{d}}
\newcommand{\norm}[1]{\left\|#1\right\|}
\newcommand{\abs}[1]{\left|#1\right|}
\newcommand{\inner}[2]{\langle #1, #2 \rangle}

\def\tilde{\widetilde}

\def\epsilon{\varepsilon}

\def\proj{\mbox{\rm proj}\,}

\def\lm{\lambda}

\def \N{{\rm I\!N}}
\def \R{{\rm I\!R}}

\newtheorem{theorem}{\bf Theorem}[section]
\newtheorem{proposition}[theorem]{\bf Proposition} 
\newtheorem{lemma}[theorem]{\bf Lemma}
\newtheorem{remark}[theorem]{\bf Remark} 
\newtheorem{example}[theorem]{\bf Example} 
\newtheorem{definition}[theorem]{\bf Definition} 
\newtheorem{problem}{\bf Problem}
\newtheorem{corollary}[theorem]{\bf Corollary} 
\newtheorem{assumption}{\bf Assumption} 
\numberwithin{equation}{section}
\allowdisplaybreaks[4]
\begin{document}

\title{A Coupled Nonsmooth Dynamical System:   Global Well-Posedness, Stability and Sensitivity Analysis}

\author[S. Zeng]{Shengda Zeng}
\address[S. Zeng]{National Center for Applied Mathematics in Chongqing\\
School of Mathematical Sciences, Chongqing Normal University,
Chongqing 401331, P.R. China.}
\email{shengdazeng@cqnu.edu.cn}

\author[J. Chen]{Jia Chen}
\address[J. Chen]{National Center for Applied Mathematics in Chongqing\\
School of Mathematical Sciences, Chongqing Normal University,
Chongqing 401331, P.R. China.}
\email{jiachen0314@163.com}

\author[V. T. Phat]{Vo Thanh Phat}
\address[V. T. Phat]{Department of Mathematics and Statistics, University of North Dakota, Grand Forks, North Dakota, United States.}
\email{thanh.vo.1@und.edu}

\subjclass[2020]{35R70, 34A60, 47H10, 47J35}

\date{\today}

\begin{abstract}
This paper studies a coupled nonsmooth dynamical system  in which a semilinear evolution equation is coupled with an implicit algebraic law governed by the normal cone to a time-dependent convex set. The main difficulty is that the algebraic variable is not given by an explicit feedback, but must be recovered from a state-dependent normal cone relation. Using a transformed variable, we recast each frozen algebraic law as a variational inequality over a convex set. Under a strongly pseudomonotone, Lipschitz  continuous hypothesis, this frozen problem has a unique algebraic response and admits a sensitivity estimate for the state-to-control map without a projection-contraction argument. We also give verifiable sufficient conditions for this hypothesis, including a weighted strongly monotone construction and a standard Lipschitz-smallness condition, and point out that the framework allows strongly pseudomonotone nonmonotone frozen operators. The coupled system is then reduced to a semilinear evolution equation with a single state variable. By combining $C_0$-semigroup estimates with a Bielecki fixed-point argument, we prove global existence and uniqueness of mild solution pairs on finite time intervals and establish Hadamard well-posedness through explicit continuous-dependence estimates. We further derive an incremental exponential stability criterion in the dissipative semigroup regime and prove continuity of the parameter-to-solution map with respect to the initial datum and an external parameter. A reduced contact-mechanics example illustrates the variational inequality formulation and an explicit projection residual that can be used as a starting point for projection- and Newton-type inner solvers after discretization.

\end{abstract}
	\keywords{Nonsmooth dynamical system; Variational inequality; Strong pseudomonotonicity; Stability; Sensitivity Analysis; Differential inclusion}

\maketitle

\begin{center}
	{\bf Dedicated to Professor Simeon Reich on the occasion of his 80th birthday}
\end{center}

\section{Introduction}\label{sec:introduction}

Coupled dynamical systems with unilateral constraints, saturation effects, and implicit constitutive laws form a broad class of nonsmooth evolution models in which a differential equation interacts with a variational inequality or a normal cone inclusion. They arise naturally in contact mechanics, sweeping processes, constrained control, and differential variational frameworks \cite{Moreau1977, Aubin1984, LiuZengMotreanu2016, BrogliatoTanwani2020, TangEtAl2020}. A common feature of such models is that the auxiliary variable cannot be prescribed as an explicit feedback of the state; instead, it is determined by a multivalued relation depending on time and on the evolving state trajectory.

Time-dependent normal cone inclusions and sweeping-type systems have been studied through several complementary approaches. The theory of moving convex sets and differential inclusions gives a general framework for nonsmooth constraints \cite{Moreau1977, Aubin1984}, while projection-based formulations are central in implicit sweeping and abstract VI models \cite{AS2019,JV2019,AlshariefJourani2025}. Recent work has also addressed implicit and nonconvex sweeping processes, prox-regular history-dependent inclusions, and controlled sweeping systems with nonsmooth sets \cite{KrejciMonteiroRecupero2021, NacrySofonea2022, PinhoFerreiraSmirnov2023}. Broader evolutionary inclusions governed by time-dependent maximal monotone relations provide an additional functional-analytic background for coupled nonsmooth dynamics \cite{SeifertTrostorffWaurick2022}. Recent work on controlled polyhedral, integro-differential, and history-dependent sweeping processes has further enlarged this viewpoint \cite{Henrion2023,Bouach2022JDE,Bouach2022JOTA,BouachHaddadThibault2022SIAM,Bouach2024ProxRegular,Godoy2025}. Motivated by these developments, we consider a coupled system in which the state satisfies a semilinear evolution equation and the algebraic variable is governed by an implicit normal cone relation involving a bounded coercive operator and a time-dependent convex set.

Most solvability theories for coupled evolution and variational systems are built on compactness arguments for solution sets or on monotonicity methods. Such approaches are well suited to existence theory, but they do not automatically provide a single-valued state-to-control representation with quantitative estimates. In the present problem this representation is essential, because the normal cone relation depends on both the state and the algebraic variable itself. Recovering a unique algebraic response with an explicit Lipschitz bound is the key step toward uniqueness, continuous dependence on the initial datum, and sensitivity with respect to external parameters in the mild-solution setting.

Two difficulties are central:
\begin{enumerate}
    \item[(i)] \textit{Implicit nonsmooth coupling.} The algebraic variable is determined through a normal cone inclusion rather than an explicit constitutive law, so the evolution equation cannot be closed unless one first proves that the implicit relation induces a well-defined single-valued map.
     \item[(ii)] \textit{Quantitative dependence estimates.} For Hadamard well-posedness and sensitivity analysis, it is not enough to solve the frozen variational problem pointwise; one also needs uniform Lipschitz bounds with respect to the state variable and a reduced dynamics compatible with semigroup estimates.
\end{enumerate}

The present paper develops a direct Hadamard well-posedness theory and stability method for this coupled nonsmooth model which contains several important and useful mathematical models as special cases, such as, differential variational inequality (see, Chen et al.~\cite{ChenXJ2022SIOPT,Chen2013SIOPT}, Pang et al.~\cite{PangStewart2008,PangStewart2009}), differential algebraic  equations (see Brogliato and Tanwani~\cite{BrogliatoTanwani2020}, Du et al. ~\cite{SIMAA}) and so forth. The key step is a variational-inequality formulation of the implicit normal cone relation based on the transformed variable $y=Pu$. For each frozen pair $(t,x)$, the algebraic law becomes $\mathrm{VI}(K(t),F_{t,x})$ with $F_{t,x}(y)=P^{-1}y-h(t,x,P^{-1}y)$. Strong pseudomonotonicity gives single-valued solvability through \cite{KimVuongKhanh2016} and yields a sensitivity estimate without requiring a relaxation parameter or a smallness-tuned contraction. We record weighted strongly monotone, strong-monotonicity, and Lipschitz-smallness conditions as convenient sufficient conditions for the frozen VI hypothesis, and a nonmonotone strongly pseudomonotone class shows that the standing hypothesis is not confined to monotone operators. This gives a single-valued state-to-control map and allows the coupled system to be treated as a reduced semilinear evolution equation. The main contributions are as follows:
\begin{itemize}
    \item We establish an equivalence between the implicit normal cone inclusion and a canonical frozen variational inequality for the transformed variable $y=Pu$.
    \item Under a strongly pseudomonotone, Lipschitz, and finite-dimensionally weakly continuous algebraic-operator hypothesis, we prove existence, uniqueness, and Lipschitz continuity of the induced state-to-control map.
    \item We provide explicit verification mechanisms for the algebraic hypothesis:  positive scalar weights of strongly monotone operators can yield genuinely nonmonotone strongly pseudomonotone examples, and a Lipschitz-smallness condition on the feedback gives a concrete strong-monotonicity criterion.
    \item We reduce the coupled problem to a semilinear evolution equation with a single state variable and prove global existence and uniqueness of mild solutions on finite time horizons by a Bielecki fixed-point argument combined with $C_0$-semigroup estimates.
    \item We derive quantitative continuous-dependence estimates with respect to the initial datum, establish global Hadamard well-posedness, and obtain an incremental exponential stability criterion in the dissipative semigroup regime.
    \item We formulate a parameterized perturbation framework for the coupled system and prove continuity of the associated parameter-to-solution map.
    \item We illustrate the abstract assumptions on a reduced contact model, where the VI route applies for all feedback strengths $b_u\ge0$ and yields explicit stability constants and a residual equation that can serve as a starting point for projection- and Newton-type solver design after discretization.
\end{itemize}

The paper is organized as follows. Section~\ref{sec:preliminaries} reviews the tools from convex and functional analysis  used throughout the paper. Section~\ref{sec:system} introduces the coupled model, the standing assumptions, and the projection formulation of the implicit relation, and develops the variational inequality formulation together with a strongly pseudomonotone solvability and sensitivity analysis. Sections~\ref{sec:globalex} and~\ref{sec:stability} are devoted to the analysis of global existence, Hadamard well-posedness, and incremental stability. Section~\ref{sec:perturbation} develops the sensitivity analysis of the parameter-to-solution map. Section \ref{sec:contact_example} illustrates the abstract framework with a reduced contact-mechanics model.  Finally, Section~\ref{sec:concl} concludes the paper with a brief summary and outlook. 

\section{Preliminaries}\label{sec:preliminaries}

We collect the projection, variational inequality, Bielecki norm, and
semigroup facts used below. Throughout, $H$ is a separable real Hilbert
space with inner product $\inner{\cdot}{\cdot}$ and norm $\norm{\cdot}$, and
$\mathbb B_r(x)$ denotes its closed ball. Standard background is available in
\cite{BauschkeCombettes2011,RockafellarWets1998,Brezis,Pazy1983,Zeidler1986}.
For $r,a\in\R$, set
\begin{equation*}
 r_+:=\max\{r,0\},\qquad
 \Phi_a(T):=\int_0^T e^{a(T-s)}\,\ud s.
\end{equation*}

\subsection{Convex Analysis}
For a nonempty, closed, and convex set $C\subset H$, let
$d_C(x):=\inf_{y\in C}\norm{x-y}$. Its metric projection $\proj_C:H\to C$
is uniquely characterized by
\begin{equation*}
    \norm{x - \proj_C(x)} = d_C(x).
\end{equation*}
\begin{lemma}[\bf variational characterization of the projection]\label{lem:proj_vi_char} 
Let $C \subset H$ be nonempty, closed, and convex. For any $z \in H$ and $y \in H$,
\begin{equation}
    y = \proj_C(z)
    \iff
    \big(y \in C \ \text{and} \ \inner{z-y}{v-y} \le 0, \ \forall v \in C\big).
\end{equation}
\end{lemma}

\begin{lemma}[\bf non-expansiveness of projection]\label{lem:non_expansive}
For every nonempty, closed, and convex $C\subset H$,
\begin{equation}
    \norm{\proj_C(x) - \proj_C(y)} \le \norm{x - y}.
\end{equation}
\end{lemma}
  
For a closed convex $C\subset H$, its normal cone is
\begin{equation}\label{def:normal_cone}
    N_C(x) = \{ \xi \in H : \inner{\xi}{y - x} \le 0, \quad \forall y \in C \}.
\end{equation}
for $x\in C$, and $N_C(x):=\emptyset$ otherwise.
\begin{proposition}[\bf projection-normal cone equivalence]\label{thm:equivalence_prelim}
Let $C \subset H$ be closed and convex, and let $\rho > 0$. For any $y, \xi \in H$, the normal cone inclusion is equivalent to the projection equation:
\begin{equation}
    \xi \in N_C(y) \iff y = \proj_C(y + \rho \xi).
\end{equation}
\end{proposition}

For a nonempty closed convex $K\subset H$, an operator $F:H\to H$ is
\emph{weakly continuous on $K$} if its restriction to the intersection of
$K$ with every finite-dimensional subspace is continuous for the weak
topology. We use the following standard generalized monotonicity notions
\cite{KS90,KimVuongKhanh2016,ZM95}:
\begin{itemize}
\item $F$ is {\em monotone} if
$
\langle F(x)-F(y), x-y\rangle \geq 0 \quad \text{for all }\; x, y \in H,
$
and {\em strongly monotone} with modulus $m_F>0$ if
$
\langle F(x)-F(y), x-y\rangle \geq m_F\norm{x-y}_H^2 \quad \text{for all }\; x, y \in H.
$
\item $F$ is {\em pseudomonotone} if
$
\langle F(y),x-y\rangle \geq 0 \Longrightarrow \langle F(x),x-y\rangle \geq 0 \quad \text{for all }\; x, y \in H.
$
It is {\em strongly pseudomonotone} with modulus $m_F>0$ if the right-hand
side is strengthened to
$
\langle F(x),x-y\rangle \geq m_F\norm{x-y}^2_H.
$
\end{itemize}
Strong monotonicity implies strong pseudomonotonicity, and monotonicity implies
pseudomonotonicity. The following weighted construction will be used later.

\begin{example}\label{prop:weighted_strong_pseudomono} Let $C\subset H$ be nonempty. Suppose $G:H\to H$ is strongly monotone on $C$ with modulus $\mu>0$, and let $a:C\to\R$ be a function satisfying
\begin{equation*}
    a(y)\ge M>0 \quad \text{for all } y\in C,
\end{equation*}
where $M>0$ is a constant. Define $F:H\to H$ on $C$ by
\begin{equation*}
    F(y)=a(y)G(y),\qquad y\in C.
\end{equation*}
Then $F$ is strongly pseudomonotone on $C$ with modulus $M\mu$.
\end{example}
\begin{proof}
Let $y_1,y_2\in C$. Suppose that $\inner{F(y_1)}{y_2-y_1}\ge 0$, 
then  $\inner{G(y_1)}{y_2-y_1}\ge 0$ because $a(y_1)>0$. Since $G$ is strongly monotone with modulus $\mu>0$, we have 
\begin{equation*}
    \inner{G(y_2)}{y_2 - y_1}
    =
    \inner{G(y_1)}{y_2 - y_1}
    +\inner{G(y_2)-G(y_1)}{y_2-y_1}
    \ge\mu\norm{y_2-y_1}^2,
\end{equation*}
which yields the following estimates
\begin{equation*}
    \inner{F(y_2)}{y_2-y_1}
    =
    a(y_2)\inner{G(y_2)}{y_2-y_1}
    \ge M\mu\norm{y_2-y_1}^2,
\end{equation*}    
which completes the proof. 
\end{proof}
The problem of finding a vector $u^* \in K \subset H$ satisfying
\begin{equation}\label{VI}
\dotproduct{F(u^*),u-u^*}\geq 0 \quad \text{for all } u \in K
\end{equation}
is called a {\em variational inequality} (VI, for short). We denote this problem by $\text{\rm VI}(K,F)$ and its solution set by $\text{\rm Sol}(K,F)$. The following result establishes the sensitivity of solutions to variational inequalities.

\begin{lemma}[\bf solution sensitivity for strongly pseudomonotone VIs]\label{lem:vi_sensitivity}
Let $C\subset H$ be nonempty, closed, and convex. Let $F_1:H\to H$ be strongly pseudomonotone on $C$ with modulus $\gamma>0$, and let $F_2:H\to H$ be arbitrary. If $y_1$ solves $\mathrm{VI}(C,F_1)$ and $y_2$ solves $\mathrm{VI}(C,F_2)$, then
\begin{equation}\label{eq:vi_sensitivity}
    \gamma\norm{y_1-y_2}\le\norm{F_1(y_2)-F_2(y_2)}.
\end{equation}
\end{lemma}
\begin{proof}
If $y_1=y_2$, then \eqref{eq:vi_sensitivity} holds trivially. Hence, assume that $y_1\neq y_2$. Since $y_1$ is a solution of $\mathrm{VI}(C,F_1)$, we have 
    $\inner{F_1(y_1)}{y_2-y_1}\ge 0.$ 
Invoking the strong pseudomonotonicity of $F_1$, we obtain
\begin{equation*}
    \inner{F_1(y_2)}{y_2-y_1}\ge\gamma\norm{y_2-y_1}^2.
\end{equation*}
On the other hand, since $y_2$ solves $\mathrm{VI}(C,F_2)$, we have  $
    \inner{F_2(y_2)}{y_1-y_2}\ge 0.$ 
Combining the last two inequalities, we deduce
\begin{equation*}
    \gamma\norm{y_2-y_1}^2
    \le
    \inner{F_1(y_2)-F_2(y_2)}{y_2-y_1}
    \le
    \norm{F_1(y_2)-F_2(y_2)}\,\norm{y_2-y_1}.
\end{equation*}
Finally, dividing both sides by $\norm{y_2-y_1}>0$ yields \eqref{eq:vi_sensitivity}.
\end{proof}

\subsection{Real and Functional Analysis}
For $I=[0,T]$, equip $C(I;H)$ with
$\norm{v}_{C(I;H)}:=\max_{t\in I}\norm{v(t)}$. The Bielecki norm is
\begin{equation}
    \norm{v}_\lambda := \max_{t \in I} e^{-\lambda t} \norm{v(t)}, \quad \text{for a tuning parameter } \lambda > 0, \; v \in C(I;H).
\end{equation} 
\begin{lemma}[\bf equivalence of Bielecki and maximum norms]\label{lem:bielecki_equiv}
Let $\lambda>0$, $T>0$, $I = [0,T]$ and $v\in C(I;H)$. Then
\begin{equation}
    e^{-\lambda T}\norm{v}_{C(I;H)} \le \norm{v}_\lambda \le \norm{v}_{C(I;H)}.
\end{equation}
In particular, $\norm{\cdot}_\lambda$ is equivalent to the standard maximum norm on $C(I;H)$. 
\end{lemma}

A $C_0$-semigroup $\{T(t)\}_{t\ge0}$ satisfies $T(0)=I$,
$T(t+s)=T(t)T(s)$, and $T(t)x\to x$ as $t\downarrow0$. Its generator is
\begin{equation*}
 D(A):=\left\{x:\lim_{t\downarrow0}\frac{T(t)x-x}{t}\text{ exists}\right\},
 \qquad Ax:=\lim_{t\downarrow0}\frac{T(t)x-x}{t}.
\end{equation*}
We use the following standard consequences from \cite{Pazy1983}.

\begin{lemma}[\bf exponential boundedness of $C_0$-semigroups]\label{lem:semigroup_bound_prelim}
Let $A:D(A)\subset H\to H$ generate a strongly continuous semigroup $\{e^{At}\}_{t\ge 0}$ on $H$. Then there exist constants $M\ge 1$ and $\omega\in\R$ such that
\begin{equation}
    \norm{e^{At}}\le M e^{\omega t} \quad \text{for all } \; t\ge 0.
\end{equation}
Moreover, for each $x\in H$, the orbit $t\mapsto e^{At}x$ is continuous on $[0,\infty)$. 
\end{lemma}
Write $L^1(0,T;H)$ for the Bochner-integrable $H$-valued functions.
\begin{lemma}[\bf continuity of semigroup convolutions]\label{lem:semigroup_convolution_prelim}
Let $A$ generate the semigroup $\{e^{At}\}_{t\ge 0}$ on $H$, and let $g\in L^1(0,T;H)$. Then the semigroup convolution defined by 
$$
(\mathcal I g)(t):=\int_0^t e^{A(t-s)}g(s)\,\ud s,\qquad t\in[0,T],
$$
is well defined as a Bochner integral and belongs to $C([0,T];H)$.

\end{lemma}

Let $X$ be a separable metric space and $T>0$. A mapping $G:[0,T]\times X\to H$ is called a {\em Carath\'{e}odory mapping} if, for every $z\in X$, the map $t\mapsto G(t,z)$ is strongly measurable, and, for almost every $t\in[0,T]$, the map $z\mapsto G(t,z)$ is continuous. 
The following theorem gives the superposition result for Carath\'eodory mappings. 
\begin{lemma}
[\bf superposition for Carath\'{e}odory mappings]\label{lem:caratheodory_superposition_prelim}
Let $X$ be a separable metric space, let $G:[0,T]\times X\to H$ be a Carath\'{e}odory mapping, and let $z\in C([0,T];X)$. Then $t\mapsto G(t,z(t))$ is strongly measurable. If, in addition, $t\mapsto \norm{G(t,z(t))}$ is dominated by a function in $L^1(0,T)$, 
then $G(\cdot,z(\cdot))\in L^1(0,T;H)$.  
\end{lemma}

\section{System Modeling and Equivalent Transformation}\label{sec:system}
We impose the following assumptions on the operators and data.
\begin{assumption}[\bf standing assumption] \rm \label{ass:operators} \textbf{}
\begin{itemize}
 \item[\bf (i)] The state operator $A: D(A) \subset H \to H$ is the infinitesimal generator of a strongly continuous semigroup $\{e^{At}\}_{t \ge 0}$ on $H$.
\item[\bf (ii)] The operator $P: H \to H$ is linear, bounded, self-adjoint, and coercive. In particular, $P$ is invertible with bounded inverse, and $P^{-1}$ is strongly monotone.
\item[\bf (iii)] The moving set $K: [0, T] \rightrightarrows H$ is a set-valued mapping defined from $[0, T]$ to $H$ with nonempty, closed, and convex values. 
\item[\bf (iv)] The nonlinear mappings $f, h: [0, T] \times H \times H \to H$ are Carath\'{e}odory mappings. In addition, the inhomogeneity of the state equation is integrable at the origin:
    \begin{equation}
        t \mapsto f(t,0,0) \in L^1(0,T;H).
    \end{equation}
\end{itemize}
\end{assumption}
We consider the following coupled system in mild state form and implicit
variational form. It is a Hilbert-space normal-cone counterpart of
differential variational inequality models and related nonsmooth couplings
\cite{PangStewart2008,LiuZengMotreanu2016,TangEtAl2020,
BrogliatoTanwani2020,ZengDuTimoshin2026}.
\begin{problem}\label{prob:primal} \rm For $T>0$, find a pair $(x,u)\in C([0,T];H)\times C([0,T];H)$ such that
\begin{numcases}{}
    x(t) = e^{At}x_0 + \int_0^t e^{A(t-s)} f(s, x(s), u(s)) \, \ud s, \quad t \in [0, T], \label{eq:state_mild} \\
    u(t) \in h(t, x(t), u(t)) - N_{K(t)}(P u(t)), \quad \text{for a.e. } t \in (0,T). \label{eq:inclusion}
\end{numcases}
\end{problem}
When convenient, we shall informally refer to the differential relation
\begin{equation*}
\dot{x}(t) = A x(t) + f(t, x(t), u(t)),
\qquad x(0)=x_0,
\end{equation*}
as the {\em strong form} representation associated with the mild equation \eqref{eq:state_mild}.

\medskip To establish the connection between the implicit inclusion and the state equation, we first reformulate the normal cone relation as a parameterized variational inequality (VI). This approach is inspired by the framework of differential variational inequalities and related evolutionary variational models; see, for example, \cite{PangStewart2008, LiuZengMotreanu2016, TangEtAl2020}. By the definition of normal cones given in  \eqref{def:normal_cone}, the inclusion \eqref{eq:inclusion} can be equivalently rewritten as the following variational inequality: For almost everywhere $t \in (0,T),$ we have 
\begin{equation}\label{eq:vi_u}
 Pu(t) \in K(t) \quad \text{and }\;     \inner{h(t, x(t), u(t)) - u(t)}{v - P u(t)} \le 0, \quad \forall v \in K(t).
\end{equation}
Passing to the transformed variable $y=Pu$ gives the canonical variational inequality
\begin{equation}\label{eq:vi_canonical}
    \text{Find } y\in K(t) \text{ such that }
    \inner{F_{t,x}(y)}{v-y}\ge 0,\quad \forall v\in K(t),
\end{equation}
where
\begin{equation}\label{eq:F_def}
    F_{t,x}(y):=P^{-1}y-h(t,x,P^{-1}y), \quad y \in H.
\end{equation}
Hence, in order to decouple the system and solve the state equation, it is necessary to establish that, for any given state $x \in H$ and time $t$, the frozen variational inequality introduced below admits a unique solution $u_x(t)$, and that the state-to-control mapping $x \mapsto u_x(t)$ is uniformly Lipschitz continuous. The following assumption provides a sufficient condition.

\begin{assumption}[\bf strong pseudomonotonicity assumption]\label{ass:strong_pseudomono} \rm  There are $\gamma >0$ and $L_F>0$ such that for every $t\in[0,T]$ and every $x\in H$ the operator $F_{t,x}$ defined as in \eqref{eq:F_def}
is strongly pseudomonotone with modulus $\gamma$   on $K(t)$ and Lipschitz continuous with modulus $L_F$ on $H$. 
\end{assumption}

\begin{remark}[\bf discussion on Assumption \ref{ass:strong_pseudomono}]\label{remark:assstrong} \rm  The class of strongly pseudomonotone operators is quite broad and does not impose a restrictive assumption. We observe the following:
\begin{itemize}
    \item[\bf (i)] Suppose that Assumption \ref{ass:operators} holds, and  there exists $\beta_h >0$ such that   
    \begin{equation}\label{Liph}
    \norm{h(t, x, u_1) - h(t, x, u_2)} \le \beta_h \norm{u_1 - u_2} \quad \text{for all }\; t \in [0,T], x \in H, u_1, u_2 \in H,
    \end{equation}
    and $\beta_h \norm{P^{-1}} < m_P$, where $m_P$ denotes the coercivity modulus of $P^{-1}$. We show that  $F_{t,x}$ is strongly monotone. Indeed, for any $y_1, y_2 \in H$, $d:=y_1-y_2$, we have  
    \begin{align*}
    \inner{F_{t,x}(y_1)-F_{t,x}(y_2)}{d}
    &=
    \inner{P^{-1}d}{d}
    -\inner{h(t,x,P^{-1}y_1)-h(t,x,P^{-1}y_2)}{d} \\
    &\ge
    m_P\norm{d}^2
    -\beta_h\norm{P^{-1}}\norm{d}^2
    =
    \bigl(m_P-\beta_h\norm{P^{-1}}\bigr)\norm{d}^2.
\end{align*}
Consequently, $F_{t,x}$ in \eqref{eq:F_def} is  strongly pseudomonotone. The Lipschitz continuity of $F_{t,x}$ is satisfied due to \eqref{Liph}, and thus   Assumption  \ref{ass:strong_pseudomono} holds. 
   \item[\bf (ii)]    Assumption~\ref{ass:strong_pseudomono}  includes frozen operators that are strongly pseudomonotone but not monotone. Indeed, let $K(t)\equiv K\subset H$ be bounded, nonempty, closed, and convex. Choose $R_K>0$ with $K\subset \mathbb{B}_{R_K}(0)$, and write $\Pi_R:=\proj_{\mathbb{B}_{R_K}(0)}$. Let $A\in\mathcal L(H)$ be coercive with modulus $\gamma_A>0$, let $b\in H$, and let $g:H\to\R$ be Lipschitz continuous on $\mathbb B_{R_K}(0)$ and satisfy $g\ge M>0$ there. Define
\begin{equation*}
    \widetilde F(y):=g(\Pi_R y)\bigl(A\Pi_R y+b\bigr),\qquad y\in H.
\end{equation*}
With the feedback realization
\begin{equation*}
    h(t,x,u):=u-\widetilde F(Pu).
\end{equation*}
one has $F_{t,x}=\widetilde F$, and on $K$,
\begin{equation*}
    F_{t,x}(y)=g(y)(Ay+b),\qquad y\in K.
\end{equation*}
By Example~\ref{prop:weighted_strong_pseudomono}, $F_{t,x}$ is strongly pseudomonotone on $K$ with modulus $M\gamma_A$. Since $\Pi_R$ is nonexpansive and both $g$ and $y\mapsto Ay+b$ are bounded and Lipschitz on $\mathbb B_{R_K}(0)$, their product after composition with $\Pi_R$ is globally Lipschitz on $H$. Thus the remaining regularity in Assumption~\ref{ass:strong_pseudomono} also holds. This class need not be monotone: for $H=\R$, $A=1$, $b=0$, $g(y)=2+\sin y$, and $K=[0,3\pi/2]$,
\begin{equation*}
    F(y)=(2+\sin y)y,\qquad
    F'(y)=2+\sin y+y\cos y,
\end{equation*}
so $F'(\pi)=2-\pi<0$. Hence the frozen law is strongly pseudomonotone but not monotone on $K$, while the truncation above preserves the global Lipschitz requirement. 
\end{itemize}
\end{remark}

\begin{proposition}[\bf unique existence  of state-to-control maps]\label{lem:unique_u}
Suppose Assumptions~\ref{ass:operators} and \ref{ass:strong_pseudomono} hold. For any fixed $t \in [0,T]$ and $x \in H$, there exists a unique $u \in H$, denoted by $u_x(t)$, satisfying \eqref{eq:vi_u}.
\end{proposition} 
\begin{proof} Fix $t\in[0,T]$ and $x\in H$. The Lipschitz continuity in
Assumption~\ref{ass:strong_pseudomono} implies the finite-dimensional weak
continuity required in \cite[Theorem~2.1]{KimVuongKhanh2016}, because weak
and norm topologies coincide on finite-dimensional subspaces. That theorem,
together with strong pseudomonotonicity and the closed convexity of $K(t)$,
therefore gives a unique solution $y_x(t)$ of \eqref{eq:vi_canonical}.
Since $P$ is invertible, $u_x(t):=P^{-1}y_x(t)$ is the unique solution of
\eqref{eq:vi_u}.
    
\end{proof}

We next introduce an additional assumption to ensure that the solution mapping established in Proposition~\ref{lem:unique_u} is Lipschitz continuous.
\begin{assumption}[\bf state-Lipschitz continuous  feedback]\label{ass:h_lipschitz_x} \rm \rm There exists  $\alpha_h>0$    such that
    for all $t \in [0, T]$ and all $x_1, x_2, u \in H$,
    the mapping $h$ satisfies:
    \begin{equation}\label{Lipassump}
     \norm{h(t, x_1, u) - h(t, x_2, u)} \le \alpha_h \norm{x_1 - x_2}. 
    \end{equation}
\end{assumption}

\begin{proposition}[\bf Lipschitz continuity of the state-to-control map]\label{lem:lipschitz_u}
In the setting of Proposition~\ref{lem:unique_u}, suppose in addition that   Assumption~\ref{ass:h_lipschitz_x} is satisfied. Then we have 
\begin{equation}\label{eq:LU_def}
    \norm{u_{x_1}(t)-u_{x_2}(t)}\le L\norm{x_1-x_2} \quad \text{for all }\; t\in [0,T], x_1, x_2 \in H, \quad \text{where }\; 
    L:=\frac{\alpha_h\norm{P^{-1}}}{\gamma}.
\end{equation}
\end{proposition}
\begin{proof}
Fix $t \in [0,T]$. Let $y_i:=y_{x_i}(t)
:=Pu_{x_i}(t)$ for  $i=1,2$. Then $y_1$ and $y_2$ solve $\mathrm{VI}(K(t),F_{t,x_1})$ and $\mathrm{VI}(K(t),F_{t,x_2})$, respectively. Applying Lemma~\ref{lem:vi_sensitivity} with $C=K(t)$, $F_1=F_{t,x_1}$, and $F_2=F_{t,x_2}$, we have 
\begin{align*}
    \gamma\norm{y_1-y_2}
    &\le\norm{F_{t,x_1}(y_2)-F_{t,x_2}(y_2)}\\
    &=\norm{h(t,x_2,P^{-1}y_2)-h(t,x_1,P^{-1}y_2)}
    \le \alpha_h \norm{x_1-x_2}, 
\end{align*}
which implies that $\norm{y_1-y_2}\le(\alpha_h/\gamma)\norm{x_1-x_2}$, and thus \eqref{eq:LU_def} is verified due to the invertibility of $P$. 

\end{proof}

The next result shows that the implicitly defined state-to-control mapping $u$, associated with continuous trajectories $x$, is itself continuous under the following additional assumption.
\begin{assumption}[\bf time regularity assumptions]\label{ass:time_regularity} \textbf{} \rm 
\begin{itemize} 
    \item[\bf (i)] For every $R\ge 0$, there exists a continuous function $\omega_{K,R}: \R_+ \to \R_+$ with $\displaystyle\lim_{r \to 0^+}\omega_{K,R}(r) = 0$  such that
    \begin{equation}
        \norm{\proj_{K(t_1)}(z) - \proj_{K(t_2)}(z)} \le \omega_{K,R}(|t_1 - t_2|), \quad \forall t_1, t_2 \in [0, T], \ \forall z \in \mathbb{B}_R(0).
    \end{equation}
    \item[\bf (ii)] For every $R\ge 0$, there exists a continuous function $\omega_{h,R}: \R_+ \to \R_+$ with $\displaystyle\lim_{r \to 0^+}\omega_{h,R}(r) = 0$  such that
    \begin{equation}
        \norm{h(t_1, x, u) - h(t_2, x, u)} \le \omega_{h,R}(|t_1 - t_2|),
    \end{equation}
    for all $t_1, t_2 \in [0, T]$ and all $x,u\in \mathbb{B}_R(0)$.
\end{itemize}
\end{assumption}

We begin with the following technical lemma, which will be used in the proof of the next result. 
\begin{lemma}\label{techni}
Let $\gamma>0$ and let $\{a_m\}$, $\{b_m\}$, and $\{s_m\}$ be sequences of  real numbers satisfying
$\gamma s_m^2 \le a_m + b_m s_m$ 
for every $m \in \N$. If
$a_m \to 0$ and 
$b_m \to 0$ as $m \to \infty$ 
then $
s_m \to 0$ as $m \to \infty.$ 
\end{lemma}

\begin{proof}
Using the fact that 
$$
\left(\sqrt{\gamma} s_m - \frac{b_m}{\sqrt{\gamma}} \right)^2 \geq 0 \quad \text{for all }\; m \in \N, 
$$
we obtain 
$$ 
b_m s_m
\leq 
\frac{\gamma}{2}s_m^2+\frac{b_m^2}{2\gamma} \quad \text{for all }\; m \in \N, 
$$
which implies that 
$$
\gamma s_m^2 \leq a_m + b_m s_m 
\leq
a_m+\frac{\gamma}{2}s_m^2+\frac{b_m^2}{2\gamma} \quad \text{for all }\; m \in \N.
$$
Hence,
$$
s_m^2
\le
\frac{2}{\gamma}a_m+\frac{b_m^2}{\gamma^2}  \quad \text{for all }\; m \in \N.
$$
Since $a_m\to0$ and $b_m\to0$, the right-hand side converges to zero, and thus  $s_m^2\to 0$. The proof is complete. 

\end{proof}

\begin{proposition}[\bf continuity of the implicit variable along continuous trajectories]\label{lem:continuity_u} Suppose that Assumptions \ref{ass:operators}, \ref{ass:strong_pseudomono}, \ref{ass:h_lipschitz_x}, and \ref{ass:time_regularity} hold. Let $x \in C([0,T];H)$, and define $u:[0,T]\to H$ by
$$
u(t):=u_{x(t)}(t) \quad \text{for all }\; t \in [0,T],
$$
where $u_{x(t)}(t)$ denotes the unique mapping corresponding to $x(t) \in H$ and $t\in[0,T]$, as constructed in Proposition \ref{lem:unique_u}. Then $u \in C([0,T];H).$ 
\end{proposition}
\begin{proof} To show that $u \in C([0,T]; H)$, we pick any $t \in [0,T],$ and an arbitrary sequence $t_m \to t$ as $m \to \infty.$  We need to verify that $u(t_m) \to u(t)$ in $H$. To proceed, for each $m \in \N$, we define 
$$
  y:=Pu(t), \quad y_m:=Pu(t_m), \quad   \tilde y_m:=\proj_{K(t_m)}(y),  \quad \text{and }\; \hat y_m:=\proj_{K(t)}(y_m).
$$
Then $y\in K(t)$ solves $\mathrm{VI}(K(t),F_{t,x(t)})$, while $y_m\in K(t_m)$ solves $\mathrm{VI}(K(t_m),F_{t_m,x(t_m)})$ for all $m \in \N$. We divide the proof into the following claims.

\medskip \noindent 
{\bf Claim 1. } {\em The sequence $\{\tilde y_m\}$ converges to $y$ as $m \to \infty$.}

\medskip 
Indeed, since $y=\proj_{K(t)}(y)$ and $y$ is fixed, Assumption~\ref{ass:time_regularity}{\bf (i)}, with $R:=\norm{y}$, implies that
\begin{equation}\label{eq:tilde_y_converges}
    \norm{\tilde y_m-y}
    =\norm{\proj_{K(t_m)}(y)-\proj_{K(t)}(y)}
    \le\omega_{K,R}(|t_m-t|),
\end{equation}
for all sufficiently large $m \in \N$, which verifies Claim 1.

\medskip \noindent 
{\bf Claim 2. } {\em The sequence $\{y_m\}$ is bounded, and we have}
\begin{equation}\label{eq:first_term_continuity}
 \norm{[F_{t_m,x(t_m)}-F_{t,x(t)}](\tilde y_m)}  \to 0  \quad \text{as }\; m \to \infty. 
\end{equation} 
Since $y_m \in \mathrm{Sol}(K(t_m),F_{t_m,x(t_m)})$ for all $m \in \N$, we have 
$$
 \inner{F_{t_m,x(t_m)}(y_m)}{\tilde y_m-y_m}\ge 0 \quad \text{for all }\; m \in \N. 
$$
Combining the latter with the  strong pseudomonotonicity of each $F_{t_m, x(t_m)}$ with modulus $\gamma>0$ yields 
\begin{equation}\label{eq:continuity_spm_start}
    \gamma\norm{\tilde y_m-y_m}^2
    \le \inner{F_{t_m,x(t_m)}(\tilde y_m)}{\tilde y_m-y_m}
    \le \norm{F_{t_m,x(t_m)}(\tilde y_m)}\,\norm{\tilde y_m-y_m} \quad \text{for all }\; m \in \N, 
\end{equation}
which implies that
\begin{equation}\label{eq:ym_local_bound_by_F}
    \norm{\tilde y_m-y_m}
    \le \frac{1}{\gamma}\norm{F_{t_m,x(t_m)}(\tilde y_m)}\quad \text{for all }\; m \in \N.  
\end{equation}
Because $\tilde y_m\to y$ and $x(t_m)\to x(t)$, the sequences $\{\tilde y_m\}$, $\{P^{-1}\tilde y_m\}$, and $\{x(t_m)\}$ are bounded. Choose a radius $R_0$ containing these points together with $x(t)$ and $P^{-1}y$. Then using Assumptions~\ref{ass:strong_pseudomono}, \ref{ass:h_lipschitz_x}, and \ref{ass:time_regularity}, we have 
\begin{align*}
    \norm{F_{t_m,x(t_m)}(\tilde y_m)}
    &\le \norm{P^{-1}\tilde y_m}
       +\norm{h(t_m,x(t_m),P^{-1}\tilde y_m)-h(t_m,x(t),P^{-1}\tilde y_m)}\\
    &\quad +\norm{h(t_m,x(t),P^{-1}\tilde y_m)-h(t,x(t),P^{-1}\tilde y_m)}
       +\norm{h(t,x(t),P^{-1}\tilde y_m)}\\
    &\le \norm{P^{-1}}\norm{\tilde y_m}
       +\alpha_h\norm{x(t_m)-x(t)}
       +\omega_{h,R_0}(|t_m-t|)\\
    &\quad +\norm{h(t,x(t),P^{-1}y)}
       +\norm{F_{t,x(t)}(\tilde y_m)-F_{t,x(t)}(y)}
       +\norm{P^{-1}(\tilde y_m-y)}\\
    &\le \norm{P^{-1}}\norm{\tilde y_m}
       +\alpha_h\norm{x(t_m)-x(t)}
       +\omega_{h,R_0}(|t_m-t|)\\
    &\quad +\norm{h(t,x(t),P^{-1}y)}
       +(L_F+\norm{P^{-1}})\norm{\tilde y_m-y},
\end{align*}
for all $m \in \N$,  which implies that the sequence $ \left\{\norm{F_{t_m,x(t_m)}(\tilde y_m)}\right\}$ is bounded. Combining this with \eqref{eq:ym_local_bound_by_F}, we deduce that $\{\tilde y_m-y_m\}$ is bounded, and then $\{y_m\}$ is bounded as well. We may consequently choose $R_h >0$ such that for all sufficiently large $m \in \N$, the points $y$, $\tilde y_m$, $y_m$, $P^{-1}\tilde y_m$, $P^{-1}y_m$,   $x(t_m)$ and $x(t)$ are in   $\mathbb{B}_{R_h}(0)$. 
Due to Assumptions \ref{ass:h_lipschitz_x} and \ref{ass:time_regularity}, we have the following estimates 
\begin{align*}
\norm{[F_{t_m,x(t_m)}-F_{t,x(t)}](\tilde y_m)}  &=\norm{h(t,x(t),P^{-1}\tilde y_m)-h(t_m,x(t_m),P^{-1}\tilde y_m) } \\
& \leq \norm{h(t,x(t),P^{-1}\tilde y_m)-h(t,x(t_m),P^{-1}\tilde y_m) }\\ & \qquad    + \norm{h(t,x(t_m),P^{-1}\tilde y_m)-h(t_m,x(t_m),P^{-1}\tilde y_m) } \\
& \leq \alpha_h \norm{x(t)-x(t_m)}+ \omega_{h,R_h}(|t_m -t|) 
\end{align*}
for all $m \in \N$. Combining the latter with the fact that $x \in C([0,T];H)$, $t_m \to t$, and $\displaystyle\lim_{r \to 0^+}\omega_{h,R_h}(r) = 0$, we deduce that $ \norm{[F_{t_m,x(t_m)}-F_{t,x(t)}](\tilde y_m)}  \to 0$ as $m \to \infty$. 

\medskip \noindent 
{\bf Claim 3. } {\em The sequence $\{y_m - \tilde y_m \}$  converges to $0$ as $m \to \infty$.}

\medskip 
Since $y_m\in K(t_m)$ and $\{y_m\} \subset \mathbb{B}_{R_h}(0)$, Assumption~\ref{ass:time_regularity}{\bf (i)} gives us the estimates 
\begin{equation}\label{eq:hat_y_close}
    \norm{\hat y_m-y_m}
    =\norm{\proj_{K(t)}(y_m)-\proj_{K(t_m)}(y_m)}
    \le\omega_{K,R_h}(|t_m-t|),
\end{equation}
for all sufficiently large $m \in \N$, and thus $\hat{y}_m - y_m \to 0$. Since  $y\in \mathrm{Sol}(K(t),F_{t,x(t)})$, we have 
\begin{equation*}
    \inner{F_{t,x(t)}(y)}{\hat y_m-y}\ge 0 \quad \text{for all }\; m \in \N,
\end{equation*}
which implies that 
\begin{equation}\label{eq:y_to_ym_term}
    \inner{F_{t,x(t)}(y)}{y-y_m}
    =  \inner{F_{t,x(t)}(y)}{y-\hat{y}_m} + \inner{F_{t,x(t)}(y)}{\hat{y}_m-y_m} \leq \norm{F_{t,x(t)}(y)}\,\norm{\hat y_m-y_m}  \quad \text{for all }\; m \in \N.
\end{equation}
Moreover, Assumption~\ref{ass:strong_pseudomono} gives us the following estimates
\begin{align*}
    \inner{F_{t,x(t)}(\tilde y_m)-F_{t,x(t)}(y)}{y-y_m}
    & \leq L_F\norm{\tilde y_m-y}\norm{y-y_m}\nonumber \\
    & \leq L_F\norm{\tilde y_m-y} \left(\norm{y-\tilde y_m}+\norm{\tilde y_m-y_m}\right)\nonumber\\
    & \leq L_F\norm{\tilde y_m-y}^2 + L_F \norm{\tilde y_m-y}\norm{\tilde y_m-y_m},
\end{align*}
for all $m \in \N$. Combining this with \eqref{eq:y_to_ym_term}, we have  
\begin{equation}\label{eq:lipschitz_F_term}
\inner{F_{t,x(t)}(\tilde y_m) }{y-y_m} \leq \norm{F_{t,x(t)}(y)}\,\norm{\hat y_m-y_m} +  L_F\norm{\tilde y_m-y}^2 + L_F \norm{\tilde y_m-y}\norm{\tilde y_m-y_m}
\end{equation}
By Assumption \ref{ass:strong_pseudomono}, we have 
\begin{align}\label{eq:innerFxt}
  \inner{F_{t,x(t)}(\tilde y_m)}{\tilde y_m-y}
 &   \leq \norm{F_{t,x(t)}(\tilde y_m)}\,\norm{\tilde y_m-y} \nonumber \\
 & \leq \left(\norm{F_{t,x(t)}(\tilde y_m) - F_{t,x(t)}(y) } + \norm{F_{t,x(t)}(y)} \right)   \norm{\tilde y_m-y} \nonumber \\
 & \leq L_F \norm{\tilde y_m-y}^2 + \norm{F_{t,x(t)}(y)}\norm{\tilde y_m-y} 
\end{align}
for all $m \in \N$. Combining \eqref{eq:lipschitz_F_term} and \eqref{eq:innerFxt}, we have 
$$
\inner{F_{t,x(t)}(\tilde y_m)}{\tilde y_m-y_m} \leq \norm{F_{t,x(t)}(y)}\left(\norm{\hat y_m-y_m}+ \norm{\tilde y_m -y} \right) +  2L_F\norm{\tilde y_m-y}^2 + L_F \norm{\tilde y_m-y}\norm{\tilde y_m-y_m}
$$
for all $m \in \N.$ Combining the latter with \eqref{eq:continuity_spm_start}, we have 
\begin{align}
    \gamma\norm{\tilde y_m-y_m}^2 & \leq \dotproduct{F_{t_m,x(t_m)}(\tilde y_m) - F_{t,x(t)}(\tilde y_m), \tilde y_m - y_m} + \dotproduct{F_{t,x(t)} (\tilde y_m), \tilde y_m - y_m} \notag \\
    &\le  \norm{[F_{t_m,x(t_m)}-F_{t,x(t)}](\tilde y_m)}\norm{\tilde y_m-y_m} + \inner{F_{t,x(t)}(\tilde y_m)}{\tilde y_m-y_m} \notag\\
    &\leq a_m + b_m \norm{\tilde y_m - y_m} \label{eq:continuity_split},
\end{align}
for all $m \in \N$, where 
$$
a_m:= \norm{F_{t,x(t)}(y)}\left(\norm{\hat y_m-y_m}+ \norm{\tilde y_m -y} \right) +  2L_F\norm{\tilde y_m-y}^2 ,
$$
and 
$$
 b_m:= L_F\norm{\tilde{y}_m - y } +  \norm{[F_{t_m,x(t_m)}-F_{t,x(t)}](\tilde y_m)}.
$$
Putting $s_m:= \norm{y_m - \tilde{y}_m}$ and combining with Claims 1, 2 and \eqref{eq:continuity_split}, we deduce that 
$$
\gamma s_m^2 \leq a_m + b_m s_m \quad \forall m \in \N, \; a_m \to 0, \; \text{and }\; b_m \to 0.
$$
Applying Lemma \ref{techni}, we can conclude that $s_m \to 0$, or $y_m - \tilde y_m \to 0$  as $m \to \infty$. By combining this with Claim 1, we must have $y_m \to y$ as $m \to \infty$, which allows us to say that $u(t_m)\to u(t)$ since $P$ is invertible. Therefore,  $u\in C([0,T];H)$, which completes the proof.  
\end{proof}

\medskip 
As demonstrated by the preceding analysis, Problem \ref{prob:primal} can be decoupled by reformulating the normal cone inclusion as a parameterized variational inequality and exploiting the fundamental properties of variational inequalities. Alternatively, the system can also be decoupled by expressing the normal cone inclusion as an equivalent Lipschitz equation. Indeed, by applying the projection--normal cone equivalence established in Proposition \ref{thm:equivalence_prelim}, Problem \ref{prob:primal} can be reformulated as the following parameterized projection system.
\begin{proposition}[\bf parameterized projection formulation]\label{thm:equivalent}
Suppose Assumption  \ref{ass:operators}   holds, and let $\rho>0$ be a positive parameter. A pair $(x, u)\in C([0,T];H)\times C([0,T];H)$ is a mild solution pair to Problem \ref{prob:primal} if and only if it satisfies
\begin{numcases}{}
    x(t) = e^{At}x_0 + \int_0^t e^{A(t-s)} f(s, x(s), u(s)) \, \ud s, \quad t \in [0, T], \label{eq:explicit_x} \\
    P u(t) = \proj_{K(t)} \big( P u(t) + \rho (h(t, x(t), u(t)) - u(t))\big), \quad \text{for a.e. } t \in (0, T). \label{eq:explicit_u}
\end{numcases}
\end{proposition}
\begin{proof} It follows from Proposition \ref{thm:equivalence_prelim} that  for any $t \in (0,T)$, the normal cone inclusion  
$$
 h(t, x(t), u(t)) -u(t) \in  N_{K(t)}(P u(t))
$$
is equivalent to the projection equation
$$
 P u(t) = \proj_{K(t)}\big(P u(t) + \rho (h(t,x(t),u(t)) - u(t))\big) 
$$
for any $\rho>0$. The inclusion in Problem \ref{prob:primal} is imposed for a.e. $t\in(0,T)$, and the pointwise equivalence just proved applies at each such time; hence the normal cone inclusion and the projection equation are equivalent as a.e. statements, which completes the proof. 

\end{proof}

For fixed $(t,x)$, the change of variables $y=Pu$ turns
\eqref{eq:explicit_u} into
\begin{equation}\label{fixpoint2}
    y=\mathcal T_{t,x}(y):=
    \proj_{K(t)}\big((I-\rho P^{-1})y+\rho h(t,x,P^{-1}y)\big).
\end{equation}
By Lemma~\ref{lem:non_expansive},
\begin{equation}\label{eq:estiT}
  \norm{\mathcal T_{t,x}(y_1)-\mathcal T_{t,x}(y_2)}
    \le \norm{(I-\rho P^{-1})(y_1-y_2)}
        +\rho\norm{h(t,x,P^{-1}y_1)-h(t,x,P^{-1}y_2)}
\end{equation}
for all $y_1, y_2 \in H.$ This motivates the following assumption.
\begin{assumption}[\bf Lipschitz and smallness assumptions] \label{ass:lipschitz_smallness}  \rm There exists $\beta_h>0$ such that
    for all $t \in [0, T]$ and all $x, u_1, u_2 \in H$,
    the mapping $h$ satisfies:
    \begin{equation}\label{Lipassump2}
     \norm{h(t, x, u_1) - h(t, x, u_2)} \leq \beta_h \norm{u_1 - u_2},
    \end{equation}
    and there exists a parameter $\rho>0$ such that
    \begin{equation}\label{smalless}
        q_\rho:=\norm{I-\rho P^{-1}}+\rho \beta_h\norm{P^{-1}}<1.
    \end{equation}
\end{assumption}
Under Assumption~\ref{ass:lipschitz_smallness}, \eqref{eq:estiT} gives
\begin{equation*}
 \norm{\mathcal T_{t,x}(y_1)-\mathcal T_{t,x}(y_2)}
 \le q_\rho\norm{y_1-y_2}.
\end{equation*}
Thus Banach's theorem yields the same unique frozen solution as the VI route.
The relation between the two verification mechanisms is recorded next.

\begin{remark}[\bf comparison of the two algebraic hypotheses]\rm
Under Assumption~\ref{ass:operators}, Assumption
\ref{ass:lipschitz_smallness} is strictly stronger than Assumption
\ref{ass:strong_pseudomono}. Indeed, \eqref{smalless} is equivalent to
$\beta_h\norm{P^{-1}}<m_P:=\min\sigma(P^{-1})$. If the latter inequality
fails, then
\begin{equation*}
    q_\rho
    \ge \abs{1-\rho m_P}+\rho \beta_h\norm{P^{-1}}
    \ge 1-\rho m_P+\rho \beta_h\norm{P^{-1}}
    \ge1
\end{equation*}
for every $\rho>0$. Conversely, suppose that $\beta_h \norm{P^{-1}} < m_P$. Putting $M_P:=\norm{P^{-1}}=\max\sigma(P^{-1})$ and   $\rho:=2/(m_P+M_P)$ gives us 
\begin{equation*}
    q_{\rho}
    \leq 
    \frac{M_P-m_P+ 2 \beta_h M_P}{M_P+m_P} < \frac{M_P-m_P+2m_P}{M_P+m_P}  = 1.
\end{equation*}
Hence \eqref{smalless} holds exactly under the stated spectral inequality,
which implies Assumption~\ref{ass:strong_pseudomono} by
Remark~\ref{remark:assstrong}{\bf (i)}. The conclusions of Propositions
\ref{lem:unique_u}, \ref{lem:lipschitz_u}, and \ref{lem:continuity_u}
therefore remain valid under the stronger smallness hypothesis.
\end{remark}

\section{Global Existence and Uniqueness via Semigroup Theory}\label{sec:globalex} 
After decoupling the algebraic inclusion \eqref{eq:inclusion} in Problem~\ref{prob:primal}, we can write $
u(t)=u_{x(t)}(t)$ for $t\in[0,T]$,
by the uniqueness result established in Proposition~\ref{lem:unique_u}. Substituting this  into \eqref{eq:state_mild}, we obtain a reduced nonlinear evolution equation, whose existence theory can be treated by standard semigroup arguments as in~\cite{JV2019}. 

\medskip 
First, we define the {\em reduced nonlinear mapping} $\tilde{f}: [0, T] \times H \to H$ as follows:
\begin{equation}\label{reduced}
    \tilde{f}(t, x) = f(t, x, u_x(t)) \quad \text{for all }\; t \in [0,T], x \in H,
\end{equation}
where $f$ is defined in Assumption \ref{ass:operators}, and $u_x : (t,x) \mapsto u_x(t)$ is the mapping constructed in Proposition \ref{lem:unique_u}. The global existence and uniqueness of solutions to Problem~\ref{prob:primal} are established in this section under the following additional assumption. 
\begin{assumption}[\bf Lipschitz continuity of the state equation]\label{ass:f_lipschitz} \rm 
There exist constants $\alpha_{f}>0$ and $\beta_{f}>0$ such that, for all $t\in[0,T]$ and all $x_1,x_2,u_1,u_2\in H$,
\begin{equation*}
    \norm{f(t,x_1,u_1)-f(t,x_2,u_2)}
    \le \alpha_{f}\norm{x_1-x_2}+\beta_{f}\norm{u_1-u_2}.
\end{equation*}
\end{assumption}
\begin{remark}[\bf A nontrivial class of admissible data] \rm
The following affine data provide an admissible class under the smallness
condition:
\begin{equation*}
f(t,x,u)=F_xx+F_uu+b_f(t),
\qquad
h(t,x,u)=H_xx+H_uu+b_h(t),
\end{equation*}
where $F_x,F_u,H_x,H_u\in\mathcal L(H)$, $b_f\in L^1(0,T;H)$,
$b_h$ is uniformly continuous, and $\proj_{K(t)}$ varies uniformly on bounded
sets, provided that
\begin{equation*}
 \norm{H_u}\norm{P^{-1}}<m_P.
\end{equation*}
This inequality verifies the algebraic hypothesis through
Assumption~\ref{ass:lipschitz_smallness} and
Remark~\ref{remark:assstrong}{\bf (i)}. Separately,
Remark~\ref{remark:assstrong}{\bf (ii)} gives a nonmonotone admissible
frozen algebraic class satisfying Assumption~\ref{ass:strong_pseudomono}
directly; it is not part of the displayed affine smallness example.
 
\end{remark}

Under Assumption \ref{ass:f_lipschitz} together with previous Assumptions \ref{ass:operators}, \ref{ass:strong_pseudomono}, \ref{ass:h_lipschitz_x}, \ref{ass:time_regularity}, the reduced nonlinear mapping \eqref{reduced} is a Carath\'eodory mapping and is Lipschitz continuous with respect to the state variables. This result is established in the following lemma.

\begin{lemma}[\bf reduced mappings are Carath\'eodory and Lipschitz continuous] \label{lem:lipschitz_reduced} Suppose that Assumptions \ref{ass:operators}, \ref{ass:strong_pseudomono}, \ref{ass:h_lipschitz_x},  \ref{ass:time_regularity}, and   \ref{ass:f_lipschitz} hold. Then the reduced mapping $(t,x)\mapsto\tilde{f}(t, x)$ as in \eqref{reduced} is Carath\'{e}odory on $[0,T]\times H$ and uniformly Lipschitz continuous with respect to the state $x$. In particular,
\begin{equation}\label{reduceLip}
    \norm{\tilde f(t,x_1)-\tilde f(t,x_2)}
    \le \bigl(\alpha_f + \beta_f L\bigr)\norm{x_1-x_2} \quad \text{for all }\; x_1, x_2 \in H,
\end{equation}
where $L>0$ is the  state-to-control Lipschitz constant from  \eqref{eq:LU_def}. 
\end{lemma}
\begin{proof} Fix $x\in H$. By applying Proposition \ref{lem:continuity_u} to the constant function $x(\cdot)\equiv x$ on $[0,T]$, we conclude that the mapping $t\mapsto u_x(t)$ is continuous on $[0,T]$. Since $f$ is Carath\'{e}odory, it follows that $
    t\mapsto \tilde f(t,x)=f(t,x,u_x(t))$
is measurable. Moreover, Proposition \ref{lem:lipschitz_u} implies that, for almost every $t\in[0,T]$, the mapping $x\mapsto u_x(t)$ is continuous. Since $f$ is a Carath\'eodory mapping, the mapping $(x,u)\mapsto f(t,x,u)$ is continuous for almost every $t\in[0,T]$. Consequently, the mapping $
x\longmapsto \tilde f(t,x)=f\bigl(t,x,u_x(t)\bigr)$ 
is continuous for almost every $t\in[0,T]$. Therefore, $\tilde f$ is a Carath\'eodory mapping. 

\medskip We next verify the Lipschitz continuity of $\tilde{f}$. Fix $t\in[0,T]$ and let $x_1,x_2\in H$ be arbitrary. By the definition of the reduced mapping, we have 
\begin{equation*}
    \tilde f(t,x_i)=f\bigl(t,x_i,u_{x_i}(t)\bigr), \qquad i=1,2.
\end{equation*}
We now split the difference into the part coming from the state variable and the part coming from the implicit variable. More specifically, 
\begin{align*}
    \tilde f(t,x_1)-\tilde f(t,x_2)
    &= f\bigl(t,x_1,u_{x_1}(t)\bigr)-f\bigl(t,x_2,u_{x_2}(t)\bigr) 
    \\ &=\Bigl[f\bigl(t,x_1,u_{x_1}(t)\bigr)-f\bigl(t,x_2,u_{x_1}(t)\bigr)\Bigr] + \Bigl[f\bigl(t,x_2,u_{x_1}(t)\bigr)-f\bigl(t,x_2,u_{x_2}(t)\bigr)\Bigr].
\end{align*}
Applying the triangle inequality gives us the following estimate. 
\begin{align*}
    \norm{\tilde f(t,x_1)-\tilde f(t,x_2)}
    &\le \norm{f\bigl(t,x_1,u_{x_1}(t)\bigr)-f\bigl(t,x_2,u_{x_1}(t)\bigr)} + \norm{f\bigl(t,x_2,u_{x_1}(t)\bigr)-f\bigl(t,x_2,u_{x_2}(t)\bigr)}.
\end{align*}
Using Assumption \ref{ass:f_lipschitz}, we deduce that 
\begin{equation*}
    \norm{f\bigl(t,x_1,u_{x_1}(t)\bigr)-f\bigl(t,x_2,u_{x_1}(t)\bigr)}
    \le  \alpha_f \norm{x_1-x_2},
\end{equation*}
and
\begin{equation*}
    \norm{f\bigl(t,x_2,u_{x_1}(t)\bigr)-f\bigl(t,x_2,u_{x_2}(t)\bigr)}
    \le \beta_f \norm{u_{x_1}(t)-u_{x_2}(t)}.
\end{equation*}
Moreover, by Proposition \ref{lem:lipschitz_u}, we can find $L>0$ such that 
\begin{equation*}
    \norm{u_{x_1}(t)-u_{x_2}(t)} \le L  \norm{x_1-x_2}.
\end{equation*}
Therefore, we obtain \eqref{reduceLip}, 
which implies that  the reduced mapping $\tilde f$ is uniformly Lipschitz continuous in the state variable $x$.  The proof is complete. 
    
\end{proof}

Having established the necessary preliminary results, we now present the main theorem of this section, which guarantees the global existence and uniqueness of mild solution pairs for Problem \ref{prob:primal}.

\begin{theorem}[\bf global existence and uniqueness]\label{thm:global_existence} Under Assumptions \ref{ass:operators}, \ref{ass:strong_pseudomono}, \ref{ass:h_lipschitz_x},   \ref{ass:time_regularity}, and   \ref{ass:f_lipschitz}, each initial datum $x_0 \in H$ generates a unique  mild solution pair $(x,u)\in C([0,T];H)\times C([0,T];H)$ on $[0,T]$ for Problem \ref{prob:primal}. The state component $x$ satisfies the reduced mild equation
\begin{equation}\label{eq:mild_solution}
    x(t) = e^{At} x_0 + \int_0^t e^{A(t-s)} \tilde{f}(s, x(s)) \ud s,
\end{equation}
whereas the algebraic component is uniquely recovered through
$u(t)=u_{x(t)}(t)$ for all $t\in[0,T]$. For each frozen pair $(t,x(t))$,
the value $u_{x(t)}(t)$ is unique by Proposition~\ref{lem:unique_u}, while
the trajectory $t\mapsto u_{x(t)}(t)$ is continuous by
Proposition~\ref{lem:continuity_u}.
\end{theorem}
\begin{proof}
Let $\mathcal{X} = C([0,T]; H)$ and equip it with the Bielecki weighted norm
\begin{equation*}
    \norm{x}_\lambda = \sup_{t \in [0,T]} e^{-\lambda t} \norm{x(t)}, \quad \text{for a tuning parameter } \lambda > 0.
\end{equation*}
By Lemma \ref{lem:bielecki_equiv}, $\norm{\cdot}_\lambda$ is equivalent to $\norm{\cdot}_{C([0,T];H)}$, so $(\mathcal{X}, \norm{\cdot}_\lambda)$ is complete.  The proof of this theorem is divided into the following claims. 

\medskip\noindent
{\bf Claim 1.} {\em The integral operator $\mathcal{S}: \mathcal{X} \to \mathcal{X}$  
\begin{equation}\label{eq:operator_S}
    (\mathcal{S}x)(t) = e^{At} x_0 + \int_0^t e^{A(t-s)} \tilde{f}(s, x(s)) \ud s, \quad \forall t \in [0, T].
\end{equation}
is well-defined.}

\medskip 
Indeed, by Lemma \ref{lem:semigroup_bound_prelim}, the map $t\mapsto e^{At}x_0$ is continuous. Let $x\in \mathcal X$,  $t \mapsto u_{x(t)}(t)$ be the mapping uniquely associated with $x$ according to Proposition \ref{lem:unique_u}, and set $R:=\norm{x}_{C([0,T];H)}$. By Proposition \ref{lem:continuity_u}, the frozen control associated with the zero state, denoted by $u_0(\cdot)$, is continuous on $[0,T]$. Hence, by Weierstrass's theorem, $u_0(\cdot)$ is bounded on $[0,T]$. Using Proposition \ref{lem:lipschitz_u}, we can find a positive number $L>0$ such that 
$$
\norm{u_{x(t)}(t)-u_0(t)} \leq L\norm{x(t)} \leq LR \quad \text{for all }\;t \in [0,T],
$$
which implies that 
\begin{equation*}
    \norm{u_{x(t)}(t)}\le \norm{u_{x(t)}(t)-u_0(t)}+\norm{u_0(t)}
    \le L R+\norm{u_0(t)} \quad \text{for all }\; t\in[0,T].
\end{equation*}
Combining the latter with Assumption \ref{ass:f_lipschitz}, for almost every $t\in[0,T]$, we have 
\begin{align*}
    \norm{\tilde f(t,x(t))} 
    & \leq \norm{f(t,x(t), u_{x(t)}(t)) -f(t,0,0)} +  \norm{f(t,0,0)} \\
    & \leq   \alpha_f\norm{x(t)} + \beta_f \norm{u_{x(t)}(t)} +\norm{f(t,0,0)}  \\
    & \leq \alpha_f R + \beta_f \bigl(L R+\norm{u_0(t)}\bigr) + \norm{f(t,0,0)}
\end{align*}
The right-hand side of the above inequality belongs to $L^1(0,T)$ because $u_0$ is continuous on the compact interval $[0,T]$ and $f(\cdot,0,0)\in L^1(0,T;H)$ by Assumption \ref{ass:operators}.  Since $\tilde f$ is Carath\'eodory by Lemma  \ref{lem:lipschitz_reduced}, Lemma \ref{lem:caratheodory_superposition_prelim} shows that the integrand is strongly measurable and Bochner integrable. Lemma \ref{lem:semigroup_convolution_prelim} then implies that $\mathcal{S}$ is well-defined.

\medskip \noindent
{\bf Claim 2.} {\em There is $\lm >0$ such that  the operator $\mathcal{S}$ in \eqref{eq:operator_S} is contractive  on $(\mathcal{X},\norm{\cdot}_\lambda)$}.

\medskip 
By Lemma \ref{lem:semigroup_bound_prelim}, there exist constants $M \ge 1$ and $\omega \in \R$ such that
\begin{equation}\label{eAt}
    \norm{e^{At}} \le M e^{\omega t}  \quad \text{for all }\;  t \ge 0.
\end{equation}
Let $x_1, x_2 \in \mathcal{X}$ be two arbitrary state trajectories. For any $t \in [0, T]$, we have 
\begin{align}\label{Sesti}
    \norm{(\mathcal{S}x_1)(t) - (\mathcal{S}x_2)(t)} &= \norm{ \int_0^t e^{A(t-s)} \big( \tilde{f}(s, x_1(s)) - \tilde{f}(s, x_2(s)) \big) \ud s } \nonumber \\ 
    &\le \int_0^t \norm{e^{A(t-s)}} \norm{\tilde{f}(s, x_1(s)) - \tilde{f}(s, x_2(s))} \ud s.
\end{align}
By Lemma \ref{lem:lipschitz_reduced}, we can find $\tilde{L}>0$ such that  
\begin{equation}\label{Lipoftilde}
  \norm{\tilde f(t,x_1(t)) - \tilde f(t,x_2(t))} \leq \tilde{L}\norm{x_1(t)-x_2(t)} \quad \text{for all }\; t \in [0,T].
\end{equation}
Combining \eqref{eAt}, \eqref{Sesti} and \eqref{Lipoftilde}, we deduce that 
$$
  \norm{(\mathcal{S}x_1)(t) - (\mathcal{S}x_2)(t)} \le M \tilde{L} \int_0^t e^{\omega (t-s)} \norm{x_1(s) - x_2(s)} \ud s.
  $$
Pick a positive number $\lm >\omega + M\tilde{L}$. Noting that  $1=e^{\lambda s}e^{-\lambda s}$ and multiplying  both sides of the above inequality by $e^{-\lambda t}$, we have the following estimates 
\begin{align*}
    e^{-\lambda t} \norm{(\mathcal{S}x_1)(t) - (\mathcal{S}x_2)(t)} &\le M \tilde{L} \int_0^t e^{\omega(t-s)} e^{-\lambda t} e^{\lambda s} e^{-\lambda s} \norm{x_1(s) - x_2(s)} \ud s \\
    &= M \tilde{L} \int_0^t e^{(\omega - \lambda)(t-s)} \big( e^{-\lambda s} \norm{x_1(s) - x_2(s)} \big) \ud s.
\end{align*}
Since $e^{-\lambda s} \norm{x_1(s) - x_2(s)} \le \norm{x_1 - x_2}_\lambda$, we deduce that 
\begin{equation}\label{Sop}
  e^{-\lambda t} \norm{(\mathcal{S}x_1)(t) - (\mathcal{S}x_2)(t)} \le M \tilde{L} \norm{x_1 - x_2}_\lambda \int_0^t e^{(\omega - \lambda)(t-s)} \ud s.
\end{equation}
Since $\lambda > \omega$, we have 
\begin{equation}\label{integralest}
     \int_0^t e^{(\omega - \lambda)(t-s)} \ud s = \left[ \frac{-1}{\lambda - \omega} e^{(\omega - \lambda)(t-s)} \right]_{s=0}^{s=t} = \frac{1 - e^{(\omega - \lambda)t}}{\lambda - \omega} \le \frac{1}{\lambda - \omega}.
\end{equation}
By \eqref{Sop} and \eqref{integralest}, it follows that 
$$
   \norm{\mathcal{S}x_1 - \mathcal{S}x_2}_\lambda \le \frac{M \tilde{L}}{\lambda - \omega} \norm{x_1 - x_2}_\lambda.
$$
Combining this with the fact that $\frac{M\tilde{L}}{\lambda-\omega}<1$,  we conclude that $\mathcal{S}$ is contractive on $(\mathcal{X},\norm{\cdot}_\lambda)$. 

\medskip\noindent
{\bf Claim 3.} {\em Problem \ref{prob:primal} has a unique mild solution.}

\medskip Indeed, since $\mathcal{S}$ is contractive on $(\mathcal{X},\norm{\cdot}_\lambda)$, it follows from the Banach fixed-point theorem that $\mathcal{S}$ has a unique fixed point $x \in \mathcal{X}$. Define $u(t):=u_{x(t)}(t)$. Its frozen value is unique by Proposition~\ref{lem:unique_u}, and the resulting trajectory is continuous by Proposition~\ref{lem:continuity_u}. We verify that the pair $(x,u)$ satisfies \eqref{eq:state_mild} and \eqref{eq:inclusion} in Problem \ref{prob:primal}. It is clear that the definition of $u$  automatically satisfies \eqref{eq:inclusion}. Moreover, since $x$ is the unique fixed-point of $\mathcal{S}$ on $\mathcal{X}$, $x$ satisfies \eqref{eq:mild_solution}. Combining the latter with the fact that $\tilde f(t,x(t))=f(t,x(t),u(t))$, we have
\begin{equation*}
    x(t) = e^{At}x_0 + \int_0^t e^{A(t-s)} f(s, x(s), u(s)) \ud s, \quad t \in [0, T],
\end{equation*}
which verifies \eqref{eq:state_mild}.  In particular, the inclusion required a.e. in Problem \ref{prob:primal} is satisfied by the continuous representative constructed above, and the pointwise identity $u(t)=u_{x(t)}(t)$ holds for all $t\in[0,T]$. Evaluating \eqref{eq:state_mild} at $t=0$ also gives $x(0)=x_0$. Uniqueness of the pair follows from uniqueness of the fixed point $x$ together with uniqueness of $u(t)=u_{x(t)}(t)$ for each frozen time. The proof is complete. 

\end{proof}

\begin{remark}[\bf numerical methods for nonsmooth dynamical systems] \rm
The proof of Theorem~\ref{thm:global_existence} suggests an outer fixed-point
iteration for the mild state equation with a frozen VI solve at each step.
Modified projection methods such as \cite{kv14} are relevant to the inner
problem when their algorithmic hypotheses are verified. The results in
\cite{kha22}, including normal-map error bounds, likewise require the
constraint and regularity assumptions imposed there; the present standing
assumptions, in particular for a possibly unbounded $K(t)$, do not by
themselves establish convergence of a complete inner algorithm. At the
continuous frozen-VI level, \cite[Theorem~4.2]{KimVuongKhanh2016} gives
\begin{equation*}
    \norm{y-y_x(t)}
    \le
    \Bigl(\tfrac{L_F+1}{\gamma}+1\Bigr)
    \norm{y-\proj_{K(t)}\bigl(y-F_{t,x}(y)\bigr)},
    \qquad y\in K(t),
\end{equation*}
which controls the frozen solution error by a computable residual.

\medskip 
Alternatively, one may solve the projected residual equation
$$
R(y):=y-\proj_{K(t)}\bigl(y-F_{t,x}(y)\bigr)=0
$$
after spatial discretization, using semismooth or coderivative-based Newton
methods \cite{FacchineiPang2003,Solo14,Ul,Helmut,BorisKhanhPhat,kmp24cod,
kmptjogo,kmptmp} under suitable finite-dimensional regularity. A convergence
analysis for these solvers and for a full time-space discretization (compare
the sweeping-process discretization background in \cite{BernicotVenel2013})
is left for future work.
\end{remark}

\section{Well-Posedness and Stability Analysis}\label{sec:stability}
In the previous section, we established global existence and uniqueness of mild solution pairs. We now turn to quantitative stability properties of the coupled system on $[0,T]$. We prove continuous dependence on the initial state, derive Hadamard well-posedness, and establish an incremental exponential stability criterion in the dissipative regime. To proceed with the main results, we need the following lemma taken from \cite{gron19}.

\begin{lemma}[\bf Gronwall's inequality]\label{lem:gronwall_prelim}
Let $\phi: [0,T] \to \R_+$ be continuous and suppose that for some constants $a,b \ge 0$,
\begin{equation}
    \phi(t) \le a + b\int_0^t \phi(s)\,\ud s, \quad \forall t \in [0,T].
\end{equation}
Then
\begin{equation}
    \phi(t) \le a e^{bt}, \quad \forall t \in [0,T].
\end{equation}
\end{lemma}

\begin{theorem}[\bf continuous dependence on initial data]\label{thm:continuous_dependence}
Under Assumptions \ref{ass:operators}, \ref{ass:strong_pseudomono}, \ref{ass:h_lipschitz_x}, \ref{ass:time_regularity}, and \ref{ass:f_lipschitz}, let $(x_i,u_i)$  be mild solution pairs to  Problem \ref{prob:primal} corresponding to initial data $x_{0,i} \in H$ for $i=1,2$. Then, for every $t \in [0,T]$, we have 
\begin{equation}\label{1stestistability}
    \norm{x_1(t)-x_2(t)} \le M e^{(\omega + M\tilde{L})t} \norm{x_{0,1}-x_{0,2}},
\end{equation}
and
\begin{equation}\label{2ndestistability}
    \norm{u_1(t)-u_2(t)} \le L M e^{(\omega + M\tilde{L})t} \norm{x_{0,1}-x_{0,2}},
\end{equation}
where $M\geq 1$ and $\omega \in \R$ are semigroup growth constants as in Lemma \ref{lem:semigroup_bound_prelim}, $\tilde{L}:=\alpha_f + \beta_f L$ is the reduced Lipschitz constant from Lemma \ref{lem:lipschitz_reduced}, and $L>0$ is the state-to-control Lipschitz constant from Proposition \ref{lem:lipschitz_u}. Consequently, we have 
\begin{equation}\label{supesti1}
    \norm{x_1-x_2}_{C([0,T];H)} \le M e^{(\omega + M\tilde{L})_+T} \norm{x_{0,1}-x_{0,2}},
\end{equation}
and
\begin{equation}\label{supesti2}
    \norm{u_1-u_2}_{C([0,T];H)} \le L M e^{(\omega + M\tilde{L})_+T} \norm{x_{0,1}-x_{0,2}}.
\end{equation}
\end{theorem}

\begin{proof} For $i=1,2$, the reduced mild representation reads
\begin{equation*}
    x_i(t)=e^{At}x_{0,i}+\int_0^t e^{A(t-s)}\tilde{f}(s,x_i(s))\,\ud s.
\end{equation*}
Subtracting the two identities and using the semigroup estimate $\norm{e^{At}}\le M e^{\omega t}$, we have 
\begin{equation*}
    \norm{x_1(t)-x_2(t)}\le M e^{\omega t}\norm{x_{0,1}-x_{0,2}} + M\int_0^t e^{\omega(t-s)}\norm{\tilde{f}(s,x_1(s)) - \tilde{f}(s,x_2(s))}\,\ud s.
\end{equation*}
By Lemma \ref{lem:lipschitz_reduced}, we have 
\begin{equation}\label{ineqx(t)}
    \norm{x_1(t)-x_2(t)}
    \le M e^{\omega t}\norm{x_{0,1}-x_{0,2}}
    + M\tilde{L}\int_0^t e^{\omega(t-s)}\norm{x_1(s)-x_2(s)}\,\ud s.
\end{equation}
Next, we define 
\begin{equation*}
    d(t):=\norm{x_1(t)-x_2(t)},\qquad z(t):=e^{-\omega t}\norm{x_1(t)-x_2(t)}. 
\end{equation*}
Then, we can rewrite \eqref{ineqx(t)} as follows: 
\begin{equation*}
    z(t)\le M\norm{x_{0,1}-x_{0,2}}+M\tilde{L}\int_0^t z(s)\,\ud s.
\end{equation*}
Applying Lemma \ref{lem:gronwall_prelim} for $a:=M\norm{x_{0,1}-x_{0,2}},\; b:=M\tilde{L},\; \phi(t):=z(t)$ yields
\begin{equation*}
    z(t)\le M e^{M\tilde{L}t}\norm{x_{0,1}-x_{0,2}},
\end{equation*}
which proves \eqref{1stestistability}. Moreover, Proposition \ref{lem:lipschitz_u} implies that for each $t \in [0,T]$, we have the  estimates
\begin{equation*}
    \norm{u_1(t)-u_2(t)}
    = \norm{u_{x_1(t)}(t)-u_{x_2(t)}(t)}
    \le L\norm{x_1(t)-x_2(t)},
\end{equation*}
which verifies \eqref{2ndestistability} by combining the latter with  \eqref{1stestistability}.   By taking the supremum over $t\in[0,T]$ in inequalities \eqref{1stestistability} and \eqref{2ndestistability}, we obtain \eqref{supesti1} and \eqref{supesti2}, which completes the proof. 

\end{proof}

The main result, Theorem~\ref{thm:continuous_dependence}, yields several important consequences. The first is the global well-posedness of the system in the sense of Hadamard, stated in the following corollary.

\begin{corollary}[\bf Hadamard well-posedness]\label{cor:hadamard_wellposed} Under Assumptions \ref{ass:operators}, \ref{ass:strong_pseudomono}, \ref{ass:h_lipschitz_x}, \ref{ass:time_regularity}, and \ref{ass:f_lipschitz}, Problem \ref{prob:primal} is globally well-posed on $[0,T]$ in the sense of Hadamard. More specifically, for every initial datum $x_0\in H$, there exists a unique mild solution pair $(x,u)\in C([0,T];H)\times C([0,T];H)$. Moreover, the associated solution map
\begin{equation*}
  \mathcal{G}: H \to C([0,T];H)\times C([0,T];H), \qquad \mathcal{G}(x_0):=(x,u),
\end{equation*}
is Lipschitz continuous. In particular, for any $x_{0,1},x_{0,2}\in H$, one has
\begin{equation*}
  \norm{\mathcal{G}(x_{0,1})-\mathcal{G}(x_{0,2})}_{C([0,T];H)\times C([0,T];H)}
    \le C_T \norm{x_{0,1}-x_{0,2}},
\end{equation*}
where
\begin{equation*}
 C_T := M e^{(\omega + M\tilde{L})_+T}(1+L),
\end{equation*}
with $M$, $\omega$, $\tilde{L}$, and $L$ as defined in Theorem~\ref{thm:continuous_dependence}. Here, the product norm is given by
\begin{equation*}
  \norm{(x,u)}_{C([0,T];H)\times C([0,T];H)}:=\norm{x}_{C([0,T];H)}+\norm{u}_{C([0,T];H)}.
\end{equation*}
\end{corollary}
\begin{proof}
Existence and uniqueness follow from Theorem \ref{thm:global_existence}, while Theorem  \ref{thm:continuous_dependence} gives
\begin{equation*}
    \norm{x_1-x_2}_{C([0,T];H)} \le M e^{(\omega + M\tilde{L})_+T} \norm{x_{0,1}-x_{0,2}}
\end{equation*}
and
\begin{equation*}
    \norm{u_1-u_2}_{C([0,T];H)} \le L M e^{(\omega + M\tilde{L})_+T} \norm{x_{0,1}-x_{0,2}}.
\end{equation*}
Therefore,
\begin{align*}
    \norm{\mathcal{G}(x_{0,1})-\mathcal{G}(x_{0,2})}_{C\times C}
    &= \norm{x_1-x_2}_{C([0,T];H)} + \norm{u_1-u_2}_{C([0,T];H)} \\
    &\le M e^{(\omega + M\tilde{L})_+T}(1+L)\norm{x_{0,1}-x_{0,2}}.
\end{align*}
This proves the claimed Lipschitz continuity of the solution map and thus the Hadamard well-posedness of Problem \ref{prob:primal}.    
\end{proof}

The next consequence concerns incremental exponential stability, which is stated in the following result.

\begin{corollary}[\bf incremental exponential stability under semigroup decay]\label{cor:incremental_stability} In the setting of Theorem \ref{thm:continuous_dependence}, assume further that the linear semigroup is exponentially decaying, namely
\begin{equation*}
    \norm{e^{At}} \le M e^{-\eta t} \quad \text{for all }\; t\ge 0,
\end{equation*} 
for some   $\eta > M\tilde{L}$. Then, for any two mild solution pairs $(x_i,u_i)$ with initial states $x_{0,i}$, we have 
\begin{equation*}
    \norm{x_1(t)-x_2(t)}\le M e^{-(\eta-M\tilde{L})t}\norm{x_{0,1}-x_{0,2}},
\end{equation*}
\begin{equation*}
    \norm{u_1(t)-u_2(t)}\le L M e^{-(\eta-M\tilde{L})t}\norm{x_{0,1}-x_{0,2}},
\end{equation*}
for all $t\in[0,T]$.  
\end{corollary}
\begin{proof}
The assumption $\norm{e^{At}}\le M e^{-\eta t}$ means that the semigroup growth constant in Theorem \ref{thm:continuous_dependence} may be taken as $\omega=-\eta$. Substituting this value into the estimate from that theorem yields
\begin{equation*}
    \norm{x_1(t)-x_2(t)}
    \le M e^{-(\eta-M\tilde L)t}\norm{x_{0,1}-x_{0,2}},
\end{equation*}
and the same substitution in the state-to-control estimate gives
\begin{equation*}
    \norm{u_1(t)-u_2(t)}
    \le L M e^{-(\eta-M\tilde L)t}\norm{x_{0,1}-x_{0,2}}.
\end{equation*}
The condition $\eta>M\tilde L$ ensures that the exponent is strictly negative, so the distance between any two trajectories decays exponentially. This is precisely the incremental exponential stability statement.     
\end{proof}

\section{Sensitivity Analysis}\label{sec:perturbation}
We now allow the initial datum and the nonlinear terms to vary. Related
parameter-sensitivity questions for differential hemivariational systems are
studied in \cite{ZengMigorskiLiu2021,LiLiu2018}. Set $\Lambda:=H$, let
$(E,d_E)$ be the external parameter space, and equip $\Lambda\times E$ with
\begin{equation*}
 d_{\Lambda\times E}((\zeta_1,\eta_1),(\zeta_2,\eta_2))
 :=\norm{\zeta_1-\zeta_2}+d_E(\eta_1,\eta_2).
\end{equation*}
The operators $A,P$ and the sets $K(t)$ remain fixed; only
$f,h:[0,T]\times H\times H\times E\to H$ depend on $\eta$.
Accordingly, the object of this section is the single-valued solution map
\begin{equation*}
 \mathcal S(\zeta,\eta):=(x_{\zeta,\eta},u_{\zeta,\eta}).
\end{equation*}
The nominal proofs from Sections~\ref{sec:system}--\ref{sec:stability} are
not repeated. We retain only the ingredients that change with the
parameter: constants uniform on bounded $D\subset E$, bounded-range moduli
for $f$ and $h$, and the final difference estimate. The radius occurring in
those moduli is fixed by the uniform a priori bound for the trajectories
under comparison.

\begin{assumption}[\bf parameterized data assumptions]\label{ass:parametrized_data} \rm
The following conditions hold.
\begin{enumerate}[\bf (i)]
\item For every $\eta\in E$, the mappings
$(t,x,u)\mapsto f(t,x,u,\eta)$ and $(t,x,u)\mapsto h(t,x,u,\eta)$ are
Carath\'{e}odory.
\item For every bounded $D\subset E$, there are
$\alpha_f(D),\beta_f(D),\alpha_h(D)\ge0$ such that, for all admissible
arguments,
\begin{align*}
 \norm{f(t,x_1,u_1,\eta)-f(t,x_2,u_2,\eta)}
 &\le\alpha_f(D)\norm{x_1-x_2}+\beta_f(D)\norm{u_1-u_2},\\
 \norm{h(t,x_1,u,\eta)-h(t,x_2,u,\eta)}
 &\le\alpha_h(D)\norm{x_1-x_2}.
\end{align*}
\item For every bounded $D\subset E$ and $R\ge0$, there are moduli
$\gamma_{f,D,R},\gamma_{h,D,R}$ tending to zero at zero such that, for
$\eta_1,\eta_2\in D$ and $x,u\in\mathbb B_R(0)$,
\begin{align*}
 \norm{f(t,x,u,\eta_1)-f(t,x,u,\eta_2)}
 &\le\gamma_{f,D,R}(d_E(\eta_1,\eta_2)),\\
 \norm{h(t,x,u,\eta_1)-h(t,x,u,\eta_2)}
 &\le\gamma_{h,D,R}(d_E(\eta_1,\eta_2)).
\end{align*}
\item For every bounded $D\subset E$ and $R\ge0$, there is a modulus
$\omega_{h,D,R}$ tending to zero at zero such that
\begin{equation*}
 \norm{h(t_1,x,u,\eta)-h(t_2,x,u,\eta)}
 \le\omega_{h,D,R}(|t_1-t_2|)
\end{equation*}
for $\eta\in D$ and $x,u\in\mathbb B_R(0)$.
\item For every bounded $D\subset E$, there are $\gamma(D),L_F(D)>0$
such that
\begin{equation*}
 F^\eta_{t,x}(y):=P^{-1}y-h(t,x,P^{-1}y,\eta)
\end{equation*}
is strongly pseudomonotone on $K(t)$ with modulus $\gamma(D)$ and
Lipschitz continuous on $H$ with modulus $L_F(D)$, uniformly for
$\eta\in D$, $t\in[0,T]$, and $x\in H$.
\item For every bounded $D\subset E$, there is $C_{h,0}(D)\ge0$ such
that $\norm{h(t,0,0,\eta)}\le C_{h,0}(D)$ for $\eta\in D$ and
$t\in[0,T]$.
\item For every bounded $D\subset E$, there is $C_{f,0}(D)\ge0$ such
that $\norm{f(t,0,0,\eta)}\le C_{f,0}(D)$ for $\eta\in D$ and
$t\in[0,T]$.
\end{enumerate}
\end{assumption}

\begin{problem}[\bf parameterized coupled problem]\label{prob:parametric} \rm
Given $(\zeta,\eta)\in\Lambda\times E$, find
$(x_{\zeta,\eta},u_{\zeta,\eta})\in C([0,T];H)^2$ such that
\begin{numcases}{}
 x_{\zeta,\eta}(t)=e^{At}\zeta+\int_0^t e^{A(t-s)}
 f(s,x_{\zeta,\eta}(s),u_{\zeta,\eta}(s),\eta)\,\ud s,
 \quad t\in[0,T],\\
 u_{\zeta,\eta}(t)\in h(t,x_{\zeta,\eta}(t),u_{\zeta,\eta}(t),\eta)
 -N_{K(t)}(Pu_{\zeta,\eta}(t)),
 \quad \text{for a.e. }t\in(0,T).
\end{numcases}
\end{problem}
\begin{definition}[\bf parameterized mild solution]\label{def:parametric_mild_solution}\rm
A pair in $C([0,T];H)^2$ satisfying Problem~\ref{prob:parametric} is called a
mild solution pair.
\end{definition}

For fixed $(\eta,t,x)$, write $u_x^\eta(t)$ for the solution of the
frozen algebraic relation.
\begin{lemma}[\bf parameterized state-to-control map]\label{lem:parametrized_state_to_control}
Under Assumptions~\ref{ass:operators} and~\ref{ass:parametrized_data}, there
exists a unique $u_x^\eta(t)\in H$ satisfying
\begin{equation}\label{eq:parametrized_static_vi}
 Pu_x^\eta(t)\in K(t),\qquad
 \inner{h(t,x,u_x^\eta(t),\eta)-u_x^\eta(t)}{v-Pu_x^\eta(t)}\le0
 \quad(v\in K(t)).
\end{equation}
Moreover, for every bounded $D\subset E$,
\begin{equation*}
 \norm{u_{x_1}^\eta(t)-u_{x_2}^\eta(t)}
 \le L(D)\norm{x_1-x_2},\qquad
 L(D):=\frac{\alpha_h(D)\norm{P^{-1}}}{\gamma(D)}.
\end{equation*}
\end{lemma}
\begin{proof}
With $y=Pu$, \eqref{eq:parametrized_static_vi} is
\begin{equation}\label{eq:parametrized_vi_y}
 y\in K(t),\qquad \inner{F^\eta_{t,x}(y)}{v-y}\ge0\quad(v\in K(t)).
\end{equation}
Assumption~\ref{ass:parametrized_data}{\bf (v)} and
\cite[Theorem~2.1]{KimVuongKhanh2016} give existence and uniqueness. If
$y_i=Pu_{x_i}^\eta(t)$, Lemma~\ref{lem:vi_sensitivity} gives
\begin{equation*}
 \gamma(D)\norm{y_1-y_2}
 \le\norm{h(t,x_2,P^{-1}y_2,\eta)-h(t,x_1,P^{-1}y_2,\eta)}
 \le\alpha_h(D)\norm{x_1-x_2},
\end{equation*}
and applying $P^{-1}$ proves the estimate.
\end{proof}

\begin{corollary}[\bf uniform bounds on bounded sets]\label{cor:parametrized_state_to_control_bound}
Assume additionally Assumption~\ref{ass:time_regularity}{\bf (i)}. For every
bounded $D\subset E$ and $R\ge0$, there is $\widehat R(D,R)>0$ such that
\begin{equation*}
 \norm{u_x^\eta(t)}\le\widehat R(D,R)
 \quad(\eta\in D,\ t\in[0,T],\ \norm{x}\le R).
\end{equation*}
\end{corollary}
\begin{proof}
Put $\kappa_K:=\sup_t\norm{\proj_{K(t)}(0)}<\infty$ and
$p_t:=\proj_{K(t)}(0)$. Strong pseudomonotonicity, tested against $p_t$,
yields for $y=Pu_x^\eta(t)$
\begin{equation*}
 \norm{y}\le\kappa_K+\gamma(D)^{-1}\norm{F^\eta_{t,x}(p_t)}.
\end{equation*}
Using Assumption~\ref{ass:parametrized_data}{\bf (ii), (v), (vi)} and
\begin{equation*}
 h(t,0,P^{-1}p_t,\eta)-h(t,0,0,\eta)
 =P^{-1}p_t-\bigl(F^\eta_{t,0}(p_t)-F^\eta_{t,0}(0)\bigr),
\end{equation*}
we obtain
\begin{equation*}
 \norm{F^\eta_{t,x}(p_t)}
 \le\alpha_h(D)R+C_{h,0}(D)+(2\norm{P^{-1}}+L_F(D))\kappa_K.
\end{equation*}
Thus the claim holds with
\begin{equation*}
 \widehat R(D,R):=\norm{P^{-1}}\left[\kappa_K+
 \frac{\alpha_h(D)R+C_{h,0}(D)+(2\norm{P^{-1}}+L_F(D))\kappa_K}{\gamma(D)}
 \right].
\end{equation*}
\end{proof}

\begin{corollary}[\bf bounded-set comparison]\label{cor:parametrized_state_to_control_compare}
Let $D\subset E$ be bounded, $R\ge0$, and
$R^\sharp:=\max\{R,\widehat R(D,R)\}$. Then
\begin{equation*}
 \norm{u_{x_1}^{\eta_1}(t)-u_{x_2}^{\eta_2}(t)}
 \le\frac{\norm{P^{-1}}}{\gamma(D)}\left[
 \alpha_h(D)\norm{x_1-x_2}
 +\gamma_{h,D,R^\sharp}(d_E(\eta_1,\eta_2))\right]
\end{equation*}
for $\eta_i\in D$, $t\in[0,T]$, and $x_i\in\mathbb B_R(0)$.
Consequently, $\eta_n\to\eta$ and $x_n\to x$ imply
$u_{x_n}^{\eta_n}(t)\to u_x^\eta(t)$ for every fixed $t$.
\end{corollary}
\begin{proof}
For $y_i=Pu_{x_i}^{\eta_i}(t)$, the uniform bound places $P^{-1}y_2$ in
$\mathbb B_{R^\sharp}(0)$. Lemma~\ref{lem:vi_sensitivity} therefore gives
\begin{align*}
 \gamma(D)\norm{y_1-y_2}
 &\le\norm{F^{\eta_1}_{t,x_1}(y_2)-F^{\eta_2}_{t,x_2}(y_2)}\\
 &\le\alpha_h(D)\norm{x_1-x_2}
 +\gamma_{h,D,R^\sharp}(d_E(\eta_1,\eta_2)).
\end{align*}
Applying $P^{-1}$ proves both assertions.
\end{proof}

\begin{corollary}[\bf continuity along continuous trajectories]\label{cor:parametrized_state_to_control_time}
If $\eta\in E$ and $x\in C([0,T];H)$, then
$t\mapsto u_{x(t)}^\eta(t)$ belongs to $C([0,T];H)$.
\end{corollary}
\begin{proof}
Choose a bounded $D\ni\eta$ and a radius $R^\star$ containing the ranges of
$x$ and $u_{x(\cdot)}^\eta(\cdot)$, as supplied by
Corollary~\ref{cor:parametrized_state_to_control_bound}. For $t_m\to t$, set
\begin{equation*}
 y_m:=Pu_{x(t_m)}^\eta(t_m),\quad y:=Pu_{x(t)}^\eta(t),\quad
 \widetilde y_m:=\proj_{K(t_m)}y,\quad
 \widehat y_m:=\proj_{K(t)}y_m.
\end{equation*}
Assumption~\ref{ass:time_regularity}{\bf (i)} gives
$\widetilde y_m\to y$. The parameterized counterparts of the two estimates
that drive Proposition~\ref{lem:continuity_u} are
\begin{align*}
 \norm{\widetilde y_m-y_m}
 &\le\gamma(D)^{-1}
 \norm{F^\eta_{t_m,x(t_m)}(\widetilde y_m)},\\
 \norm{(F^\eta_{t_m,x(t_m)}-F^\eta_{t,x(t)})(\widetilde y_m)}
 &\le\alpha_h(D)\norm{x(t_m)-x(t)}
 +\omega_{h,D,R^\star}(|t_m-t|)\longrightarrow0.
\end{align*}
The first right-hand side is bounded by Assumption
\ref{ass:parametrized_data}{\bf (ii), (iv), (v)}. More precisely, after
enlarging $R^\star$ if necessary,
\begin{align*}
 \norm{F^\eta_{t_m,x(t_m)}(\widetilde y_m)}
 &\le\norm{P^{-1}}\norm{\widetilde y_m}
 +\alpha_h(D)\norm{x(t_m)-x(t)}
 +\omega_{h,D,R^\star}(|t_m-t|)\\
 &\quad+\norm{h(t,x(t),P^{-1}y,\eta)}
 +(L_F(D)+\norm{P^{-1}})\norm{\widetilde y_m-y},
\end{align*}
whose right-hand side is bounded. Thus $\{y_m\}$ is bounded;
hence, for some common radius $R_K$ and all sufficiently large $m$, the
set-continuity assumption gives
\begin{equation*}
 \norm{\widehat y_m-y_m}
 =\norm{\proj_{K(t)}(y_m)-\proj_{K(t_m)}(y_m)}
 \le\omega_{K,R_K}(|t_m-t|)\longrightarrow0.
\end{equation*}
Because $y$ solves the limiting VI and $\widehat y_m\in K(t)$,
\begin{equation*}
 \inner{F^\eta_{t,x(t)}(y)}{y-y_m}
 \le \norm{F^\eta_{t,x(t)}(y)}\norm{\widehat y_m-y_m}.
\end{equation*}
The Lipschitz property of $F^\eta_{t,x(t)}$ also yields
\begin{align*}
 \inner{F^\eta_{t,x(t)}(\widetilde y_m)}{y-y_m}
 &\le\norm{F^\eta_{t,x(t)}(y)}\norm{\widehat y_m-y_m}\\
 &\quad+L_F(D)\norm{\widetilde y_m-y}^2
 +L_F(D)\norm{\widetilde y_m-y}\norm{\widetilde y_m-y_m}.
\end{align*}
Adding the analogous estimate for
$\inner{F^\eta_{t,x(t)}(\widetilde y_m)}{\widetilde y_m-y}$ and comparing
the operators at $t_m$ and $t$ gives the same closing inequality as in
Proposition~\ref{lem:continuity_u}, with bounded-parameter constants.
Namely,
\begin{align*}
 a_m&:=\norm{F^\eta_{t,x(t)}(y)}
 \bigl(\norm{\widehat y_m-y_m}+\norm{\widetilde y_m-y}\bigr)
 +2L_F(D)\norm{\widetilde y_m-y}^2,\\
 b_m&:=L_F(D)\norm{\widetilde y_m-y}
 +\norm{(F^\eta_{t_m,x(t_m)}-F^\eta_{t,x(t)})(\widetilde y_m)}.
\end{align*}
The preceding projection and operator-difference estimates imply
$a_m,b_m\to0$. Therefore
\begin{equation*}
 \gamma(D)s_m^2\le a_m+b_ms_m,\qquad
 s_m:=\norm{y_m-\widetilde y_m}.
\end{equation*}
Lemma~\ref{techni} gives $s_m\to0$, whence $y_m\to y$ and the result follows
by the boundedness of $P^{-1}$.
\end{proof}

Combining the preceding two corollaries also shows that, whenever
$\eta_n\to\eta$ and $x_n\to x$ in $C([0,T];H)$ with uniformly bounded
ranges, the comparison estimate is uniform in time. In fact, for a common
bounded parameter set $D$ and a corresponding radius $R^\sharp$,
\begin{align*}
 &\sup_{t\in[0,T]}
 \norm{u_{x_n(t)}^{\eta_n}(t)-u_{x(t)}^\eta(t)}\\
 &\qquad\le\frac{\norm{P^{-1}}}{\gamma(D)}
 \left[\alpha_h(D)\norm{x_n-x}_{C([0,T];H)}
 +\gamma_{h,D,R^\sharp}(d_E(\eta_n,\eta))\right]\longrightarrow0.
\end{align*}
Consequently,
\begin{equation*}
 t\mapsto u_{x_n(t)}^{\eta_n}(t)\longrightarrow
 t\mapsto u_{x(t)}^\eta(t)\qquad\text{in }C([0,T];H).
\end{equation*}

Define the parameterized reduced mapping
\begin{equation*}
 \widetilde f_\eta(t,x):=f(t,x,u_x^\eta(t),\eta).
\end{equation*}
\begin{lemma}[\bf continuity of the reduced mapping]\label{lem:parametrized_reduced_stability}
Under Assumptions~\ref{ass:operators},
\ref{ass:time_regularity}{\bf (i)}, and~\ref{ass:parametrized_data},
$\widetilde f_\eta$ is Carath\'{e}odory for each $\eta$. On every bounded
$D\subset E$ it is uniformly Lipschitz in $x$, with constant
\begin{equation*}
 \widetilde L(D):=\alpha_f(D)+\beta_f(D)L(D).
\end{equation*}
Moreover, for every $R\ge0$ and $\eta_n\to\eta$ in $D$,
\begin{equation*}
 \sup_{t\in[0,T]}\sup_{x\in\mathbb B_R(0)}
 \norm{\widetilde f_{\eta_n}(t,x)-\widetilde f_\eta(t,x)}\longrightarrow0.
\end{equation*}
More precisely, with $R^\sharp:=\max\{R,\widehat R(D,R)\}$, the left-hand
side is bounded by the parameter modulus
\begin{equation}\label{eq:reduced_parameter_modulus}
 \Theta_{D,R}(r):=
 \frac{\beta_f(D)\norm{P^{-1}}}{\gamma(D)}
 \gamma_{h,D,R^\sharp}(r)+\gamma_{f,D,R^\sharp}(r),
 \qquad \Theta_{D,R}(r)\to0\quad(r\to0^+).
\end{equation}
\end{lemma}
\begin{proof}
For fixed $x$, Corollary~\ref{cor:parametrized_state_to_control_time} makes
$t\mapsto u_x^\eta(t)$ continuous, so
$t\mapsto f(t,x,u_x^\eta(t),\eta)$ is measurable. For almost every fixed
$t$, Lemma~\ref{lem:parametrized_state_to_control} and Assumption
\ref{ass:parametrized_data}{\bf (i)} make the same composition continuous in
$x$. Hence $\widetilde f_\eta$ is Carath\'{e}odory. Moreover,
\begin{align*}
 \norm{\widetilde f_\eta(t,x_1)-\widetilde f_\eta(t,x_2)}
 &\le\alpha_f(D)\norm{x_1-x_2}
 +\beta_f(D)\norm{u_{x_1}^\eta(t)-u_{x_2}^\eta(t)}\\
 &\le\widetilde L(D)\norm{x_1-x_2}.
\end{align*}
For the parameter dependence, let
$R^\sharp:=\max\{R,\widehat R(D,R)\}$. Corollary
\ref{cor:parametrized_state_to_control_compare} gives the parameter-specific
bound. To see separately where its two terms enter, write
\begin{align*}
 \norm{\widetilde f_{\eta_n}(t,x)-\widetilde f_\eta(t,x)}
 &\le\norm{f(t,x,u_x^{\eta_n}(t),\eta_n)
            -f(t,x,u_x^\eta(t),\eta_n)}\\
 &\quad+\norm{f(t,x,u_x^\eta(t),\eta_n)
            -f(t,x,u_x^\eta(t),\eta)}.
\end{align*}
The first term is controlled by $\beta_f(D)$ times the frozen-response
difference, whereas the second is controlled directly by the parameter
modulus of $f$. Consequently,
\begin{align*}
 \norm{\widetilde f_{\eta_n}(t,x)-\widetilde f_\eta(t,x)}
 &\le \frac{\beta_f(D)\norm{P^{-1}}}{\gamma(D)}
 \gamma_{h,D,R^\sharp}(d_E(\eta_n,\eta))\\
 &\quad+\gamma_{f,D,R^\sharp}(d_E(\eta_n,\eta)),
\end{align*}
uniformly on $[0,T]\times\mathbb B_R(0)$; this is precisely
\eqref{eq:reduced_parameter_modulus}.
\end{proof}

\begin{theorem}[\bf well-posedness for the parameterized problem]\label{thm:parametrized_wellposed}
Under Assumptions~\ref{ass:operators},
\ref{ass:time_regularity}{\bf (i)}, and~\ref{ass:parametrized_data}, each
$(\zeta,\eta)\in\Lambda\times E$ generates a unique mild solution pair
$(x_{\zeta,\eta},u_{\zeta,\eta})\in C([0,T];H)^2$. The state satisfies
\begin{equation}\label{eq:parametrized_problem}
 x_{\zeta,\eta}(t)=e^{At}\zeta+\int_0^t e^{A(t-s)}
 \widetilde f_\eta(s,x_{\zeta,\eta}(s))\,\ud s,
\end{equation}
and $u_{\zeta,\eta}(t)=u_{x_{\zeta,\eta}(t)}^\eta(t)$.
\end{theorem}
\begin{proof}
Fix a bounded $D\ni\eta$. Lemma~\ref{lem:parametrized_reduced_stability}
gives the Carath\'{e}odory property and the constant $\widetilde L(D)$.
Since $u_0^\eta$ is continuous and Assumption
\ref{ass:parametrized_data}{\bf (vii)} holds,
\begin{equation*}
 \norm{\widetilde f_\eta(t,0)}
 =\norm{f(t,0,u_0^\eta(t),\eta)}
 \le\beta_f(D)\norm{u_0^\eta(t)}+C_{f,0}(D),
\end{equation*}
which is integrable on $[0,T]$. Hence the operator
\begin{equation*}
 (\mathcal P_{\zeta,\eta}z)(t):=e^{At}\zeta+
 \int_0^t e^{A(t-s)}\widetilde f_\eta(s,z(s))\,\ud s
\end{equation*}
maps $C([0,T];H)$ into itself. The Bielecki estimate from
Theorem~\ref{thm:global_existence}, with
$\lambda>\omega+M\widetilde L(D)$, is
\begin{equation*}
 \norm{\mathcal P_{\zeta,\eta}z_1-\mathcal P_{\zeta,\eta}z_2}_\lambda
 \le\frac{M\widetilde L(D)}{\lambda-\omega}\norm{z_1-z_2}_\lambda,
 \qquad \frac{M\widetilde L(D)}{\lambda-\omega}<1.
\end{equation*}
Its unique fixed point solves \eqref{eq:parametrized_problem}; the continuous
algebraic component is recovered from
Corollary~\ref{cor:parametrized_state_to_control_time}. Uniqueness follows
from the same contraction and Lemma~\ref{lem:parametrized_state_to_control}.
\end{proof}

\begin{lemma}[\bf uniform a priori bounds on bounded parameter sets]\label{lem:parametrized_apriori}
Under the assumptions of Theorem~\ref{thm:parametrized_wellposed}, every
bounded $B\subset\Lambda\times E$ admits $R_B>0$ such that all corresponding
solution pairs take values in $\mathbb B_{R_B}(0)^2$.
\end{lemma}
\begin{proof}
Let $R_{\Lambda,B}:=\sup_{(\zeta,\eta)\in B}\norm{\zeta}$ and let $D_B$ be
the projection of $B$ onto $E$. Corollary
\ref{cor:parametrized_state_to_control_bound} and Assumption
\ref{ass:parametrized_data}{\bf (vii)} give
\begin{equation*}
 \norm{\widetilde f_\eta(t,0)}\le
 C_B:=\beta_f(D_B)\widehat R(D_B,0)+C_{f,0}(D_B),
\end{equation*}
and hence
$\norm{\widetilde f_\eta(t,x)}\le C_B+\widetilde L(D_B)\norm{x}$.
The mild equation therefore gives, uniformly for $(\zeta,\eta)\in B$,
\begin{align*}
 \norm{x_{\zeta,\eta}(t)}
 &\le Me^{\omega t}R_{\Lambda,B}
 +MC_B\int_0^t e^{\omega(t-s)}\,\ud s\\
 &\quad+M\widetilde L(D_B)\int_0^t e^{\omega(t-s)}
 \norm{x_{\zeta,\eta}(s)}\,\ud s.
\end{align*}
For $z(t):=e^{-\omega t}\norm{x_{\zeta,\eta}(t)}$, this becomes
\begin{equation*}
 z(t)\le MR_{\Lambda,B}+MC_B\Phi_{-\omega}(T)
 +M\widetilde L(D_B)\int_0^t z(s)\,\ud s.
\end{equation*}
Gronwall's inequality yields
\begin{equation*}
 \sup_{t\in[0,T]}\norm{x_{\zeta,\eta}(t)}
 \le R_B^x:=e^{\omega_+T}
 \bigl(MR_{\Lambda,B}+MC_B\Phi_{-\omega}(T)\bigr)
 e^{M\widetilde L(D_B)T}.
\end{equation*}
Applying Corollary~\ref{cor:parametrized_state_to_control_bound} with
$R=R_B^x$ proves the claim with
$R_B:=\max\{R_B^x,\widehat R(D_B,R_B^x)\}$.
\end{proof}

\begin{theorem}[\bf continuous dependence on $(\zeta,\eta)$]\label{thm:parametrized_continuity}
Under the assumptions of Theorem~\ref{thm:parametrized_wellposed}, if
$(\zeta_n,\eta_n)\to(\zeta,\eta)$, then the corresponding solution pairs
satisfy
\begin{equation*}
 (x_n,u_n)\longrightarrow(x,u)
 \qquad\text{in }C([0,T];H)\times C([0,T];H).
\end{equation*}
Thus the solution map
$\mathcal S(\zeta,\eta):=(x_{\zeta,\eta},u_{\zeta,\eta})$ is continuous.
\end{theorem}
\begin{proof}
The convergent parameter set is bounded, so
Lemma~\ref{lem:parametrized_apriori} places all trajectories in a common ball
$\mathbb B_{R_*}(0)$. Choose a bounded $D\subset E$ containing all
$\eta_n,\eta$ and set
\begin{equation*}
 \Delta_n:=\sup_{t\in[0,T]}\sup_{y\in\mathbb B_{R_*}(0)}
 \norm{\widetilde f_{\eta_n}(t,y)-\widetilde f_\eta(t,y)}\longrightarrow0.
\end{equation*}
In particular, Lemma~\ref{lem:parametrized_reduced_stability} gives
\begin{equation*}
 \Delta_n\le \Theta_{D,R_*}(d_E(\eta_n,\eta)),
\end{equation*}
which records all dependence on the external parameter needed below.
Subtracting the reduced mild equations gives
\begin{align*}
 x_n(t)-x(t)
 &=e^{At}(\zeta_n-\zeta)+\int_0^t e^{A(t-s)}
 \bigl[\widetilde f_{\eta_n}(s,x_n(s))
 -\widetilde f_\eta(s,x(s))\bigr]\,\ud s.
\end{align*}
Split the bracket by adding and subtracting
$\widetilde f_{\eta_n}(s,x(s))$. The common state-Lipschitz constant and
$\Delta_n$ then imply
\begin{align*}
 \norm{x_n(t)-x(t)}
 &\le Me^{\omega t}\norm{\zeta_n-\zeta}
 +M\widetilde L(D)\int_0^t e^{\omega(t-s)}
 \norm{x_n(s)-x(s)}\,\ud s\\
 &\quad+M\Delta_n\int_0^t e^{\omega(t-s)}\,\ud s.
\end{align*}
Set
\begin{equation*}
 d_n(t):=\norm{x_n(t)-x(t)},\qquad
 z_n(t):=e^{-\omega t}d_n(t),\qquad
 b:=M\widetilde L(D).
\end{equation*}
After multiplication by $e^{-\omega t}$, the preceding inequality becomes
\begin{equation*}
 z_n(t)\le M\norm{\zeta_n-\zeta}
 +M\Delta_n\int_0^t e^{-\omega s}\,\ud s
 +b\int_0^t z_n(s)\,\ud s.
\end{equation*}
Apply the convolution form of Gronwall's inequality to
$h_n(t):=M\Delta_n\int_0^t e^{-\omega s}\,\ud s$. Since
$h_n(0)=0$ and $h_n'(s)=M\Delta_n e^{-\omega s}$, it yields
\begin{equation*}
 z_n(t)\le
 M e^{bt}\norm{\zeta_n-\zeta}
 +M\Delta_n\int_0^t e^{-\omega s}e^{b(t-s)}\,\ud s.
\end{equation*}
Multiplying by $e^{\omega t}$ shows explicitly that
\begin{equation*}
 d_n(t)\le
 M e^{(\omega+M\widetilde L(D))t}\norm{\zeta_n-\zeta}
 +M\Delta_n\int_0^t
 e^{(\omega+M\widetilde L(D))(t-s)}\,\ud s.
\end{equation*}
Taking the supremum over $t\in[0,T]$ gives the existing key
parameter-difference estimate
\begin{equation*}
 \norm{x_n-x}_{C([0,T];H)}
 \le M e^{(\omega+M\widetilde L(D))_+T}\norm{\zeta_n-\zeta}
 +M\Phi_{\omega+M\widetilde L(D)}(T)\Delta_n\longrightarrow0.
\end{equation*}
Together with $\Delta_n\le\Theta_{D,R_*}(d_E(\eta_n,\eta))$, this displays
the two perturbation channels explicitly: the initial datum is propagated
by the semigroup estimate, while the external parameter enters through the
bounded-range moduli of $f$ and $h$. In particular,
\begin{align*}
 \norm{x_n-x}_{C([0,T];H)}
 &\le M e^{(\omega+M\widetilde L(D))_+T}\norm{\zeta_n-\zeta}\\
 &\quad+M\Phi_{\omega+M\widetilde L(D)}(T)
 \Theta_{D,R_*}(d_E(\eta_n,\eta)).
\end{align*}
For the algebraic variables, put
$R^\sharp:=\max\{R_*,\widehat R(D,R_*)\}$. Corollary
\ref{cor:parametrized_state_to_control_compare} and the frozen-state
Lipschitz estimate are applied to the splitting
\begin{align*}
 \norm{u_n(t)-u(t)}
 &\le\norm{u_{x_n(t)}^{\eta_n}(t)-u_{x_n(t)}^\eta(t)}
 +\norm{u_{x_n(t)}^\eta(t)-u_{x(t)}^\eta(t)}.
\end{align*}
Taking the supremum in time gives
\begin{align*}
 \norm{u_n-u}_{C([0,T];H)}
 &\le \frac{\norm{P^{-1}}}{\gamma(D)}
 \gamma_{h,D,R^\sharp}(d_E(\eta_n,\eta))
 +L(D)\norm{x_n-x}_{C([0,T];H)}\longrightarrow0.
\end{align*}
Substituting the preceding state estimate also gives the fully explicit
bound
\begin{align*}
 \norm{u_n-u}_{C([0,T];H)}
 &\le \frac{\norm{P^{-1}}}{\gamma(D)}
 \gamma_{h,D,R^\sharp}(d_E(\eta_n,\eta))\\
 &\quad+L(D)M e^{(\omega+M\widetilde L(D))_+T}
 \norm{\zeta_n-\zeta}\\
 &\quad+L(D)M\Phi_{\omega+M\widetilde L(D)}(T)
 \Theta_{D,R_*}(d_E(\eta_n,\eta)).
\end{align*}
Thus all constants and moduli in the parameterized conclusion are uniform
on the bounded set containing the convergent data.
This proves the asserted continuity.
\end{proof}

\section{A Contact-Mechanics Example}\label{sec:contact_example}
We close with a reduced boundary model of saturated normal contact. Complete
quasistatic models for a viscoelastic body normally start from a bulk
constitutive law and an equilibrium equation, together with displacement,
traction, and contact boundary conditions; see, for example, \cite{AS2019}.
Normal-cone and complementarity formulations provide a standard language for
unilateral contact reactions in nonsmooth mechanics
\cite{BrogliatoTanwani2020}. The present example is not derived from such a
bulk contact problem. Its purpose is instead to isolate a boundary feedback
law that verifies the abstract hypotheses of this paper and displays three
features: compression-only reactions, a finite pressure capacity, and an
implicit contact relation.

Let $\Gamma_C$ be a potential contact boundary of finite surface measure and
set
\begin{equation*}
    H=L^2(\Gamma_C).
\end{equation*}
For $s\in\Gamma_C$, we interpret $x(t,s)$ as a reduced normal indentation
coordinate, with the positive direction pointing toward the foundation.
Thus increasing $x$ corresponds to stronger indentation and a smaller
separation, whereas a larger gap corresponds to decreasing $x$. The model
does not introduce the geometric gap as an independent unknown. We interpret
$g(t,s)$ as a prescribed indentation demand and define the effective contact
drive
\begin{equation}\label{eq:concrete_drive}
    q(t,x):=g(t)-b_xx,
    \qquad b_x>0.
\end{equation}
Accordingly, increasing $g$ strengthens the compressive demand, while the
current indentation state reduces the unmet demand through the term $b_xx$.
The region $q(t,x)\le0$ will be called the unloaded region of this reduced
law. It may be interpreted as separation within the reduced model, but it is
not a full geometric nonpenetration condition.

The algebraic variable $u(t,s)$ is a normalized normal reaction. We choose
\begin{equation*}
    P=\beta I,
    \qquad \beta>0,
\end{equation*}
and identify
\begin{equation*}
    p(t)=Pu(t)=\beta u(t)
\end{equation*}
with the physical normal pressure, using the convention that compression is
positive. Here $\beta$ is a positive scaling from the normalized reaction to
pressure; no separate constitutive interpretation of $\beta$ is required.
The admissible pressure set is
\begin{equation}\label{eq:concrete_contact_K}
    K:=\{r\in L^2(\Gamma_C):0\le r(s)\le\kappa_0
    \text{ for a.e. }s\in\Gamma_C\},
    \qquad \kappa_0>0.
\end{equation}
The lower bound excludes tensile contact pressure. The upper bound represents
the maximum compressive pressure transferable by the deformable foundation;
once this capacity is reached, unresolved high-pressure behavior is modeled
by saturation. Hence this is not the classical unbounded Signorini pressure
cone. The set $K$ is nonempty, closed, bounded, and convex, and it is constant
in time, so Assumption~\ref{ass:time_regularity}{\bf (i)} is automatic. Its
metric projection is the pointwise projection onto $[0,\kappa_0]$.

We use the first-order reduced boundary dynamics
\begin{equation*}
    A=-\alpha I,
    \qquad \alpha>0,
\end{equation*}
and
\begin{equation}\label{eq:concrete_f}
    f(t,x,u)=\ell(t)-d_xx-d_uu,
    \qquad d_x,d_u>0,
\end{equation}
where $\ell\in L^1(0,T;H)$ is an external generalized boundary excitation in
the positive indentation direction. The term $-\alpha x$, encoded in the
generator, is the basic linear relaxation, while $-d_xx$ is an additional
bounded state-feedback term. These contributions could be combined in this
affine example, but they are kept separate to exhibit the abstract splitting
$Ax+f(t,x,u)$. Finally, $-d_uu=-(d_u/\beta)p$ is the reaction of the
compressive pressure against further indentation. Thus the strong-form
interpretation is
\begin{equation*}
    \dot x(t)=\ell(t)-(\alpha+d_x)x(t)-d_uu(t),
\end{equation*}
and the corresponding mild equation is
\begin{equation}\label{eq:concrete_state}
    x(t)=e^{-\alpha t}x_0+
    \int_0^t e^{-\alpha(t-s)}
    \bigl(\ell(s)-d_xx(s)-d_uu(s)\bigr)\,\ud s.
\end{equation}
This is a first-order boundary dynamics, not an inertial elasticity equation
and not a substitute for the bulk stress--strain system of a full contact
model. Since $e^{At}=e^{-\alpha t}I$, the semigroup constants are $M=1$ and
$\omega=-\alpha$.

The reduced implicit contact feedback is
\begin{equation}\label{eq:concrete_h}
    h(t,x,u)=g(t)-b_xx-b_uu,
    \qquad b_u\ge0,
\end{equation}
and the abstract normal-cone relation becomes
\begin{equation}\label{eq:concrete_contact_law}
    u(t)\in g(t)-b_xx(t)-b_uu(t)-N_K(\beta u(t)).
\end{equation}
The coefficient $b_u$ measures reaction self-feedback: a larger reaction
reduces the remaining contact drive. The functions $g\in C([0,T];H)$ and
$\ell\in L^1(0,T;H)$ provide, respectively, the required time regularity of
$h$ and the integrability of the state-equation inhomogeneity. Moreover,
\begin{equation*}
    \alpha_h=b_x,
    \qquad \alpha_f=d_x,
    \qquad \beta_f=d_u.
\end{equation*}

We now derive the frozen variational inequality and the associated pressure
law. Put
\begin{equation*}
    y=Pu=\beta u.
\end{equation*}
For fixed $(t,x)$, relation \eqref{eq:concrete_contact_law} is equivalent to
\begin{equation*}
    q(t,x)-\frac{1+b_u}{\beta}y\in N_K(y),
\end{equation*}
or, equivalently,
\begin{equation*}
    0\in F_{t,x}^{\rm c}(y)+N_K(y),
\end{equation*}
where
\begin{equation}\label{eq:concrete_F}
    F_{t,x}^{\rm c}(y)
    :=\frac{1+b_u}{\beta}y-q(t,x)
    =\frac{1+b_u}{\beta}y+b_xx-g(t).
\end{equation}
Hence the frozen contact problem is precisely
\begin{equation}\label{eq:concrete_vi}
    y\in K,
    \qquad
    \inner{F_{t,x}^{\rm c}(y)}{v-y}\ge0
    \quad\text{for all }v\in K.
\end{equation}

Because the feasible interval is pointwise and the coefficient of $y$ in
\eqref{eq:concrete_F} is positive, the frozen pressure has the explicit form
\begin{equation}\label{eq:concrete_pressure_law}
    p_x(t,s)=y_x(t,s)
    =\operatorname{proj}_{[0,\kappa_0]}
    \left(
    \frac{\beta}{1+b_u}q(t,x)(s)
    \right)
    \quad\text{for a.e. }s\in\Gamma_C.
\end{equation}
Along the actual trajectory, $p(t)=p_{x(t)}(t)$, and the three pointwise
contact regimes are
\begin{equation}\label{eq:concrete_contact_states}
    p(t,s)=
    \begin{cases}
        0,
        & q(t,x(t))(s)\le0,\\[1mm]
        \dfrac{\beta}{1+b_u}q(t,x(t))(s),
        & 0<q(t,x(t))(s)<\dfrac{1+b_u}{\beta}\kappa_0,\\[3mm]
        \kappa_0,
        & q(t,x(t))(s)\ge\dfrac{1+b_u}{\beta}\kappa_0.
    \end{cases}
\end{equation}
The first line is the unloaded regime (interpretable as separation only
within this reduced model), the second is active but unsaturated contact with
a linear compliance-type response, and the third is saturated contact at the
pressure capacity.

The VI verification is immediate but important. For all $y_1,y_2\in H$,
\begin{equation}\label{eq:concrete_strong_mono}
    \inner{F_{t,x}^{\rm c}(y_1)-F_{t,x}^{\rm c}(y_2)}{y_1-y_2}
    =\frac{1+b_u}{\beta}\norm{y_1-y_2}^2.
\end{equation}
Thus $F_{t,x}^{\rm c}$ is strongly monotone, and therefore strongly
pseudomonotone, with modulus
\begin{equation*}
    \gamma_{\rm c}=\frac{1+b_u}{\beta}>0,
\end{equation*}
and it is Lipschitz continuous with the same modulus
$L_F^{\rm c}=(1+b_u)/\beta$. Proposition~\ref{lem:unique_u} consequently
gives a unique frozen pressure $y_x(t)\in K$ and a unique normalized reaction
$u_x(t)=\beta^{-1}y_x(t)$ for every $b_u\ge0$.

This conclusion is stronger than the projection-contraction test in
Assumption~\ref{ass:lipschitz_smallness}. Indeed, with the natural choice
$\rho=\beta$, its contraction factor is
\begin{equation*}
    q_\rho
    =\norm{I-\rho P^{-1}}+\rho b_u\norm{P^{-1}}
    =b_u,
\end{equation*}
so that route requires $b_u<1$. By contrast,
\eqref{eq:concrete_strong_mono} holds for every $b_u\ge0$.

The same feedback coefficient also controls the sensitivity of the reaction.
Using Lemma~\ref{lem:vi_sensitivity} with the two frozen states $x_1,x_2$
gives
\begin{equation}\label{eq:concrete_LU_vi}
    \norm{u_{x_1}(t)-u_{x_2}(t)}
    \le L^{\rm c}\norm{x_1-x_2},
    \qquad
    L^{\rm c}:=\frac{b_x}{1+b_u}.
\end{equation}
Equivalently, the physical pressure satisfies
\begin{equation*}
    \norm{p_{x_1}(t)-p_{x_2}(t)}
    \le\frac{\beta b_x}{1+b_u}\norm{x_1-x_2}.
\end{equation*}
Increasing $b_u$ strengthens the reaction self-feedback and decreases the
slope from the contact drive to the reaction; it therefore reduces the
state-to-reaction sensitivity. The reduced nonlinearity has the Lipschitz
constant
\begin{equation*}
    \widetilde L_{\rm c}
    =d_x+d_uL^{\rm c}
    =d_x+\frac{d_ub_x}{1+b_u}.
\end{equation*}
Substitution of $M=1$, $\omega=-\alpha$, and $L=L^{\rm c}$ into
Corollary~\ref{cor:hadamard_wellposed} yields the concrete Hadamard constant
\begin{equation}\label{eq:concrete_CT}
    C_T^{\rm c}
    =e^{\left(d_x+\frac{d_ub_x}{1+b_u}-\alpha\right)_+T}
    \left(1+\frac{b_x}{1+b_u}\right).
\end{equation}
Consequently, for two solutions with the same data and different initial
states,
\begin{equation*}
    \norm{(x_1,u_1)-(x_2,u_2)}_{C\times C}
    \le C_T^{\rm c}\norm{x_{0,1}-x_{0,2}}.
\end{equation*}
If
\begin{equation*}
    \alpha>d_x+\frac{d_ub_x}{1+b_u},
\end{equation*}
Corollary~\ref{cor:incremental_stability} gives
\begin{align*}
    \norm{x_1(t)-x_2(t)}
    &\le
    e^{-\left(\alpha-d_x-\frac{d_ub_x}{1+b_u}\right)t}
    \norm{x_{0,1}-x_{0,2}},\\
    \norm{u_1(t)-u_2(t)}
    &\le
    \frac{b_x}{1+b_u}
    e^{-\left(\alpha-d_x-\frac{d_ub_x}{1+b_u}\right)t}
    \norm{x_{0,1}-x_{0,2}}.
\end{align*}
This is a sufficient decay estimate inherited from the abstract Lipschitz
argument. In particular, that argument uses the magnitude $d_x$ and does not
exploit the dissipative sign of $-d_xx$; the displayed exponent should not be
interpreted as an exact or optimal physical decay rate.

If either $g$ or $\ell$ depends on an external parameter and the resulting
data satisfy Assumption~\ref{ass:parametrized_data}, then
Theorems~\ref{thm:parametrized_wellposed}
and~\ref{thm:parametrized_continuity} apply to the corresponding
parameterized contact system.

Finally, the frozen VI has the $H$-valued projection residual
\begin{equation*}
    R_\tau(t,x,y)
    :=y-\proj_K\bigl(y-\tau F_{t,x}^{\rm c}(y)\bigr),
    \qquad \tau>0.
\end{equation*}
The identity $R_\tau(t,x,y)=0$ is equivalent to
\eqref{eq:concrete_vi}, and $\proj_K$ is the pointwise truncation encoded in
\eqref{eq:concrete_pressure_law}. At the continuous level this residual is an
abstract function-space quantity. A computational projection, active-set,
Newton, or semismooth Newton method first requires a spatial discretization;
Newton-type convergence additionally requires suitable finite-dimensional
regularity. No convergence theorem for such a fully discrete contact
algorithm is asserted here.

\section{Conclusion}\label{sec:concl}
We have developed a direct semigroup-based analysis for a coupled nonsmooth dynamical system consisting of a semilinear evolution equation and an implicit normal-cone algebraic relation. The main structural step was the projection reformulation of the implicit inclusion, which yielded a parameterized variational inequality and a single-valued state-to-control map under a strongly pseudomonotone frozen operator. A Lipschitz-smallness condition appears as a concrete strong-monotone verification of the algebraic hypothesis, while the analysis itself is formulated at the strongly pseudomonotone level and includes nonmonotone frozen laws. This reduction allowed us to prove global existence and uniqueness, Hadamard well-posedness with respect to the initial datum, incremental exponential stability in dissipative regimes, and continuity of the parameter-to-solution map. The reduced contact example shows how the VI formulation verifies solvability and a conservative sufficient stability rate beyond the natural contraction threshold. Several directions remain open, including parameter-dependent moving sets, perturbations of the underlying semigroup, and higher regularity of the associated solution map.


\begin{thebibliography}{99}

\bibitem{AS2019}
S. Adly and M. Sofonea, 
\emph{Time-dependent inclusions and sweeping processes in contact mechanics}, 
Zeitschrift f\"{u}r angewandte Mathematik und Physik (ZAMP), 70:39 (2019).
\url{https://doi.org/10.1007/s00033-019-1084-4}

\bibitem{AlshariefJourani2025}
L. Alsharief and A. Jourani,
\emph{Optimal control of implicit polyhedral sweeping processes involving the coderivative of the projection mapping},
SIAM Journal on Control and Optimization, 63 (2025), 2856--2886.
\url{https://doi.org/10.1137/23M1622118}

\bibitem{Aubin1984}
J.-P. Aubin and A. Cellina, 
\emph{Differential Inclusions: Set-Valued Maps and Viability Theory}, 
Grundlehren der mathematischen Wissenschaften, Vol. 264, Springer-Verlag, Berlin, (1984).
\url{https://doi.org/10.1007/978-3-642-69512-4}

\bibitem{BauschkeCombettes2011}
H. H. Bauschke and P. L. Combettes,
\emph{Convex Analysis and Monotone Operator Theory in Hilbert Spaces},
CMS Books in Mathematics, Springer, New York, 2011.
\url{https://doi.org/10.1007/978-1-4419-9467-7}

\bibitem{BernicotVenel2013}
F. Bernicot and J. Venel,
\emph{Convergence order of a numerical scheme for sweeping process},
SIAM Journal on Control and Optimization, 51 (2013), 3075--3092.
\url{https://doi.org/10.1137/120882044}

\bibitem{Bouach2022JDE}
A. Bouach, T. Haddad, and B. S. Mordukhovich,
\emph{Optimal control of nonconvex integro-differential sweeping processes},
Journal of Differential Equations, 329 (2022), 255--317.
\url{https://doi.org/10.1016/j.jde.2022.05.004}

\bibitem{Bouach2022JOTA}
A. Bouach, T. Haddad, and L. Thibault,
\emph{On the Discretization of Truncated Integro-Differential Sweeping Process and Optimal Control},
Journal of Optimization Theory and Applications, 193 (2022), 785--830.
\url{https://doi.org/10.1007/s10957-021-01991-z}

\bibitem{BouachHaddadThibault2022SIAM}
A. Bouach, T. Haddad, and L. Thibault,
\emph{Nonconvex integro-differential sweeping process with applications},
SIAM Journal on Control and Optimization, 60 (2022), 2971--2995.
\url{https://doi.org/10.1137/21M1397635}

 


\bibitem{Brezis}
H. Brezis,
\emph{Functional Analysis, Sobolev Spaces and Partial Differential Equations},
Springer, 2011.
\url{https://link.springer.com/book/10.1007/978-0-387-70914-7}

\bibitem{BrogliatoTanwani2020}
B. Brogliato and A. Tanwani,
\emph{Dynamical Systems Coupled with Monotone Set-Valued Operators: Formalisms, Applications, Well-Posedness, and Stability},
SIAM Review, 62 (2020), 3--129.
\url{https://doi.org/10.1137/18M1234795}

\bibitem{ChenXJ2022SIOPT} X. Chen and J. Shen, \emph{Dynamic stochastic variational inequalities and convergence of discrete approximation},
SIAM Journal on Optimization, 32 (2022), 2909--2937.
\url{https://doi.org/10.1137/21M145536X}

\bibitem{Chen2013SIOPT} X. Chen and Z. Wang, \emph{Convergence of regularized time-stepping methods for differential variational inequalities}, SIAM journal on optimization, 23 (2013),  1647--1671.

\bibitem{SIMAA}N. H. Du, V. H. Linh, V. Mehrmann and D. D. Thuan,
\emph{Stability and robust stability of linear time-invariant delay differential-algebraic equations}, SIAM Journal on Matrix Analysis and Applications, 34 (2013),  1631--1654.
\url{https://doi.org/10.1137/13092611}


\bibitem{FacchineiPang2003}
F.~Facchinei and J.-S. Pang,
\newblock {\em Finite-Dimensional Variational Inequalities and Complementarity Problems},
\newblock Springer, New York, 2003.
\url{https://doi.org/10.1007/b97544}

\bibitem{Bouach2024ProxRegular}
S. Gaouir, T. Haddad, and L. Thibault,
\emph{Prox-Regular Integro-Differential Sweeping Process},
Journal of Optimization Theory and Applications, 203 (2024), 1413--1438.
\url{https://doi.org/10.1007/s10957-024-02472-9}

\bibitem{Helmut} H. Gfrerer and J. V.  Outrata,  {\em On a semismooth$^*$ Newton method for solving generalized equations}, SIAM Journal on Optimization,  31 (2021) 489--517. \url{https://doi.org/10.1137/19M1257408}


\bibitem{Godoy2025}
M. Godoy, M. Torres-Valdebenito, and E. Vilches,
\emph{A Fixed-Point Approach to History-Dependent Sweeping Processes},
Journal of Optimization Theory and Applications, 206 (2025), Article 37.
\url{https://doi.org/10.1007/s10957-025-02714-4}

\bibitem{gron19} T. H. Gronwall, {\em Note on the derivatives with respect to a parameter of the solutions of a system of differential equations,} Annals of Mathematics, 20 (1919), 292--296. \url{https://doi.org/10.2307/1967124}

\bibitem{Solo14} A. F. Izmailov and  M. V. Solodov,  {\em Newton-Type Methods for Optimization and Variational Problems}, Springer, New York, 2014. \url{https://link.springer.com/book/10.1007/978-3-319-04247-3}


\bibitem{Henrion2023}
R. Henrion, A. Jourani, and B. S. Mordukhovich,
\emph{Controlled polyhedral sweeping processes: Existence, stability, and optimality conditions},
Journal of Differential Equations, 366 (2023), 408--443.
\url{https://doi.org/10.1016/j.jde.2023.04.010}

\bibitem{JV2019}
A. Jourani and E. Vilches, 
\emph{A differential equation approach to implicit sweeping processes}, 
Journal of Differential Equations, 266 (2019) 5168--5184.
\url{https://doi.org/10.1016/j.jde.2018.10.024}


\bibitem{KS90} S. Karamardian and S. Schaible,  {\em Seven kinds of monotone maps.} Journal of Optimization Theory and Applications,  66 (1990) 37--46.  \url{https://doi.org/10.1007/BF00940531}

\bibitem{KimVuongKhanh2016}
D.~S. Kim, P.~T. Vuong, and P.~D. Khanh,
\emph{Qualitative properties of strongly pseudomonotone variational inequalities},
 Optimization Letters, 10 (2016),  1669--1679. \url{https://doi.org/10.1007/s11590-015-0960-x}

 \bibitem{kha22} P. T. Kha and  P.~D. Khanh,  \emph{Variational inequalities governed by strongly pseudomonotone operators},  Optimization, 71 (2022), 1983--2004. \url{https://doi.org/10.1080/02331934.2020.1847107}

\bibitem{kv14} P.~D. Khanh and P.~T. Vuong,  \emph{Modified projection method for strongly pseudomonotone variational inequalities}, Journal of Global Optimization, 58 (2014), 341--350. \url{https://doi.org/10.1007/s10898-013-0042-5}

\bibitem{BorisKhanhPhat} P. D. Khanh, B. S. Mordukhovich, and V. T. Phat,  {\em A generalized Newton method for subgradient systems}, Mathematics of Operations Research, 48 (2023), 1811--1845. \url{https://doi.org/10.1287/moor.2022.1320}


\bibitem{kmp24cod}  P. D. Khanh, B. S. Mordukhovich, and V. T. Phat, {\em Coderivative-based Newton methods in structured nonconvex and nonsmooth optimization}, arXiv:2403.04262 (2024). \url{https://arxiv.org/abs/2403.04262}

\bibitem{kmptjogo}  P. D. Khanh, B. S. Mordukhovich, V. T. Phat, and D. B.  Tran,   {\em Generalized damped Newton algorithms in nonsmooth optimization via second-order subdifferentials},  Journal of Global Optimization, 86 (2023), 93--122. \url{https://doi.org/10.1007/s10898-022-01248-7} 

			
\bibitem{kmptmp} P. D. Khanh, B. S. Mordukhovich, V. T. Phat, and D. B.  Tran,   {\em Globally convergent coderivative-based generalized Newton methods in nonsmooth optimization},  Mathematical Programming, 205 (2024), 373--429. \url{https://doi.org/10.1007/s10107-023-01980-2}





\bibitem{KrejciMonteiroRecupero2021}
P. Krej\v{c}\'{i}, G. A. Monteiro, and V. Recupero,
\emph{Explicit and Implicit Non-convex Sweeping Processes in the Space of Absolutely Continuous Functions},
Applied Mathematics and Optimization, 84(Suppl. 2) (2021), 1477--1504.
\url{https://doi.org/10.1007/s00245-021-09801-8}

\bibitem{LiLiu2018}
X. Li and Z. Liu,
\emph{Sensitivity analysis of optimal control problems described by differential hemivariational inequalities},
SIAM Journal on Control and Optimization, 56 (2018), 3569--3597.
\url{https://doi.org/10.1137/17M1162275}

\bibitem{LiuZengMotreanu2016}
Z. Liu, S. Zeng, and D. Motreanu,
\emph{Evolutionary problems driven by variational inequalities},
Journal of Differential Equations, 260 (2016), 6787--6799.
\url{https://doi.org/10.1016/j.jde.2016.01.012}

\bibitem{Moreau1977}
J. J. Moreau, 
\emph{Evolution problem associated with a moving convex set in a Hilbert space}, 
Journal of Differential Equations, 26(3) (1977) 347--374.
\url{https://doi.org/10.1016/0022-0396(77)90085-7}

\bibitem{NacrySofonea2022}
F. Nacry and M. Sofonea,
\emph{History-dependent operators and prox-regular sweeping processes},
Fixed Point Theory and Algorithms for Sciences and Engineering, 2022 (2022), Article 5.
\url{https://doi.org/10.1186/s13663-022-00715-w}

\bibitem{PangStewart2008}
J.-S. Pang and D. E. Stewart,
\emph{Differential variational inequalities},
Mathematical Programming, 113(2) (2008), 345--424.
\url{https://doi.org/10.1007/s10107-006-0052-x}

\bibitem{PangStewart2009}
J.-S. Pang and D. E. Stewart,
\emph{Solution dependence on initial conditions in differential variational inequalities}, Mathematical Programming, 116 (2098), 429--460.
\url{https://doi.org/10.1007/s10107-007-0117-5}

\bibitem{Pazy1983}
A. Pazy, 
\emph{Semigroups of Linear Operators and Applications to Partial Differential Equations}, 
Applied Mathematical Sciences, Vol. 44, Springer-Verlag, New York, (1983).
\url{https://doi.org/10.1007/978-1-4612-5561-1}

\bibitem{PinhoFerreiraSmirnov2023}
M. do Ros\'{a}rio de Pinho, M. M. A. Ferreira, and G. Smirnov,
\emph{A Maximum Principle for Optimal Control Problems Involving Sweeping Processes with a Nonsmooth Set},
Journal of Optimization Theory and Applications, 199(1) (2023), 273--297.
\url{https://doi.org/10.1007/s10957-023-02283-4}

\bibitem{RockafellarWets1998}
R. T. Rockafellar and R. J.-B. Wets,
\emph{Variational Analysis},
Grundlehren der mathematischen Wissenschaften, Vol. 317, Springer, Berlin, 1998.
\url{https://doi.org/10.1007/978-3-642-02431-3}

\bibitem{SeifertTrostorffWaurick2022}
C. Seifert, S. Trostorff, and M. Waurick,
\emph{Evolutionary Inclusions},
in \emph{Evolutionary Equations}, Operator Theory: Advances and Applications, Vol. 287, Birkh\"{a}user, Cham, 2022, pp. 275--297.
\url{https://doi.org/10.1007/978-3-030-89397-2_17}

\bibitem{TangEtAl2020}
G.-j. Tang, J. Cen, V. T. Nguyen, and S. Zeng,
\emph{Differential variational--hemivariational inequalities: existence, uniqueness, stability, and convergence},
Journal of Fixed Point Theory and Applications, 22 (2020), Article 83.
\url{https://doi.org/10.1007/s11784-020-00814-4}

\bibitem{Ul} M. Ulbrich,  {\em Semismooth Newton Methods for Variational Inequalities and Constrained Optimization Problems in Function Spaces,} SIAM, Philadelphia, 2011. \url{https://doi.org/10.1137/1.9781611970692}


\bibitem{Zeidler1986}
E. Zeidler, 
\emph{Nonlinear Functional Analysis and its Applications: I: Fixed-Point Theorems}, 
Springer-Verlag, New York, (1986).
\url{https://doi.org/10.1007/978-1-4612-4838-5}

\bibitem{ZengDuTimoshin2026}
S. Zeng, J. Du, and S. A. Timoshin,
\emph{Controllability analysis of a class of differential-variational inequalities with nonlocal conditions},
SIAM Journal on Control and Optimization, 64 (2026), 985--1010.
\url{https://doi.org/10.1137/25M1819028}

\bibitem{ZengMigorskiLiu2021}
S. Zeng, S. Migo\'rski, and Z. Liu,
\emph{Well-Posedness, Optimal Control, and Sensitivity Analysis for a Class of Differential Variational-Hemivariational Inequalities},
SIAM Journal on Optimization, 31(4) (2021), 2829--2862.
\url{https://doi.org/10.1137/20M1351436}

\bibitem{ZM95} D. Zhu and P. Marcotte, {\em New classes of generalized monotonicity}, Journal of Optimization Theory and Applications, 87 (1995) 457--471. \url{https://doi.org/10.1007/BF02192574}

\end{thebibliography}
\end{document}